  \documentclass[10pt]{amsart}

\PassOptionsToPackage{linktocpage}{hyperref}
\usepackage[french, english]{babel}
\usepackage[utf8, latin1]{inputenc}

\usepackage{stmaryrd} %\sslash

\usepackage{amscd,amssymb,mathtools}
\usepackage{amsmath,mathrsfs,amsthm,amsfonts}
\usepackage{mathrsfs}

\usepackage[bb=px]{mathalpha}
\usepackage{bbm}
\usepackage{hyperref}
\usepackage{tikz-cd}
\newcommand{\ols}[1]{\mskip.5\thinmuskip\overline{\mskip-.5\thinmuskip {#1} \mskip-.5\thinmuskip}\mskip.5\thinmuskip} % overline short
\newcommand{\olsi}[1]{\,\overline{\!{#1}}} % overline short italic
\makeatletter
\newcommand\closure[1]{
  \tctestifnum{\count@stringtoks{#1}>1} %checks if number of chars in arg > 1 (including '\')
  {\ols{#1}} %if arg is longer than just one char, e.g. \mathbb{Q}, \mathbb{F},...
  {\olsi{#1}} %if arg is just one char, e.g. K, L,...
}

\long\def\count@stringtoks#1{\tc@earg\count@toks{\string#1}}
\long\def\count@toks#1{\the\numexpr-1\count@@toks#1.\tc@endcnt}
\long\def\count@@toks#1#2\tc@endcnt{+1\tc@ifempty{#2}{\relax}{\count@@toks#2\tc@endcnt}}
\def\tc@ifempty#1{\tc@testxifx{\expandafter\relax\detokenize{#1}\relax}}
\long\def\tc@earg#1#2{\expandafter#1\expandafter{#2}}
\long\def\tctestifnum#1{\tctestifcon{\ifnum#1\relax}}
\long\def\tctestifcon#1{#1\expandafter\tc@exfirst\else\expandafter\tc@exsecond\fi}
\long\def\tc@testxifx{\tc@earg\tctestifx}
\long\def\tctestifx#1{\tctestifcon{\ifx#1}}
\long\def\tc@exfirst#1#2{#1}
\long\def\tc@exsecond#1#2{#2}
\makeatother

\newtheorem{theorem}[subsubsection]{Theorem}
\newtheorem{lemm}[subsubsection]{Lemma}
\newtheorem{definition}[subsubsection]{Definition}
\newtheorem{remark}[subsubsection]{Remark}
\newtheorem*{remark*}{Remark}

\newtheorem{prop}[subsubsection]{Proposition}

\newtheorem{prop-def}[subsubsection]{Proposition-Definition}
\numberwithin{equation}{subsection}

\newtheorem{atheorem}{Theorem}[section]

\newtheorem{adefinition}[atheorem]{Definition}

\newcommand{\Hom}{\mathrm{Hom}}
\newcommand{\Frob}{\mathrm{Frob}}

\renewcommand{\d}{\mathrm{d}}

\renewcommand{\hat}{\widehat}
\newcommand{\vol}{\mathrm{vol}}

\newcommand{\Aut}{\mathrm{Aut}}

\newcommand{\Gal}{\mathrm{Gal}}
\newcommand{\Ind}{\mathrm{Ind}}
\newcommand{\Tr}{\mathrm{Tr}}

\newcommand{\Pic}{\mathrm{Pic}}
\newcommand{\ad}{\mathrm{Ad}}
\newcommand{\Ad}{\mathrm{Ad}}
\newcommand{\cInd}{\mathrm{c\text{-}Ind}}

\newcommand{\res}{\mathrm{res}}

\newcommand{\AAA}{\mathbb{A}}
\newcommand{\bbb}{\mathfrak{b}}
\renewcommand{\ggg}{\mathfrak{g}}

\newcommand{\nnn}{\mathfrak{n}}

\newcommand{\ppp}{\mathfrak{p}}

\newcommand{\zzz}{\mathfrak{z}}
\newcommand{\ttt}{\mathfrak{t}}

\newcommand{\ooo}{\mathcal{O}}

\newcommand{\car}{\mathfrak{c}}

\newcommand{\ago}{\mathfrak{a}}

\DeclareMathOperator{\Id}{Id}

\DeclareMathOperator{\Spec}{Spec}

\renewcommand{\leq}{\leqslant}
\renewcommand{\geq}{\geqslant}
\newcommand{\htau}{\hat \tau}

\begin{document}

\title[$\ell$-adic local systems and Higgs bundles]
{$\ell$-adic local systems and Higgs bundles: the generic case} 
%\footnote{This article corresponds to the preprint arXiv:2304.06637, which is part of an old version of arXiv:2110.13858. This latter preprint is divided into two papers, and the other one is submitted elsewhere.}

\author{Hongjie Yu}

\address{
Weizmann Institute of Science, Herzl St 234, Rehovot, Israel \newline
Current address: Morningside Center of Mathematics, Academy of Mathematics and Systems Science, Chinese Academy of Sciences, Beijing 100190, China}

\email{hongjie.yu@amss.ac.cn}

\begin{abstract}
Let $X$ be a projective smooth geometrically connected curve defined over a finite field $\mathbb{F}_q$ of cardinality $q$. Let $S$ be a non-empty finite set of closed points of $X$. Let $\closure{X}$ and $\closure{S}$ be the base change of $X$, $S$ to an algebraic closure $\closure{\mathbb{F}}_q$. We consider the set of $\closure{\mathbb{Q}}_\ell$-local systems ($\ell\nmid q$)  of rank $n$ over $\closure{X}-\closure{S}$ with prescribed tame regular semisimple and generic ramifications in $\closure{S}$. The genericity ensures that such an $\ell$-adic local system is automatically irreducible. We show that the number of these $\ell$-adic local systems fixed by Frobenius endomorphism equals the number of stable logarithmic Higgs bundles of rank $n$ and degree $e$ coprime to $n$, with a fixed residue, up to a power of $q$. In the split case, this number is equal to the number of stable parabolic Higgs bundles (with full flag structures)  with generic parabolic weights fixed by the $\mathbb{G}_m$-action defined as the dilation on Higgs fields. 
\end{abstract}
\maketitle

\section{Introduction}
Let $\Sigma$ be a compact Riemann surface and $S$ a finite set of points on $\Sigma$. The Simpson correspondence (see \cite{Simpson1} \cite{Simpson}) is a bijection between filtered stable degree $0$ $\mathbb{C}$-local systems, and filtered stable degree $0$ Higgs bundles, over $\Sigma-S$. The notion of a filtered Higgs bundle is essentially the same as a Higgs bundle with a quasi-parabolic structure. The notion of a filtered local system is a local system with filtration on each puncture. This correspondence preserves the rank and permutes residue data defined by Simpson around each puncture. Roughly speaking, it exchanges parabolic weights of parabolic Higgs bundles and eigenvalues of monodromy actions of local systems.

We are interested in the generic case with regular semisimple monodromy actions. On the local system side, it is the case where the eigenvalues of the monodromy action around each puncture are distinct, and these eigenvalues lie in a general position. The latter is a condition to guarantee that a local system with monodromy actions prescribed by these eigenvalues is automatically irreducible. The generic condition implies that semistability coincides with stability on the Higgs bundle side. In this case, under Simpson's correspondence, a stable (strongly) parabolic Higgs bundle corresponds to an irreducible local system with trivial filtration around each puncture (see the diagram on page 720 of \cite{Simpson} where we set $b=c=0$).

The main result of this article presents an identity between the number of isomorphism classes of $\ell$-adic local systems (in coefficients $\closure{\mathbb{Q}}_\ell$) over a punctured curve over a finite field $\mathbb{F}_q$  ($\ell$ is different from the characteristic of $\mathbb{F}_q$), with prescribed tame regular semisimple generic local monodromies, and, up to a power of $q$, the number of semistable logarithmic Higgs bundles with prescribed residue in the coprime case. 
When the puncture points are $\mathbb{F}_q$-rational and local monodromies are ``split", we demonstrate that this number is equal to the number of semistable parabolic Higgs bundles with a generic stability parameter fixed by a $\mathbb{G}_m$ action. This is analogue to, yet different from, the Riemann surface case. As a corollary, we positively confirm some of P. Deligne's conjectures on counting the $\ell$-adic local system in these cases.

There are several attempts to establish a similar theory to other circumstances that may have rich arithmetic applications. Notably, a $p$-adic Simpson correspondence is initiated by Faltings in \cite{Faltings}. His construction was pursued by A. Abbes, M. Gros, and T. Tsuji \cite{AGT}. There is also the work of Lan, Sheng, and Zuo \cite{LSZ1} where the role of semistability is studied. It is still an active research area. For a field over a positive characteristic, there are various versions of correspondence that can be seen as analogous to Simpson's correspondence. However, the picture is not yet clear. Especially, $\ell$-adic local systems have not entered the scene.

The current work is initially motivated by Deligne's conjectures in \cite{Deligne} and \cite{IHES}  regarding the enumeration of $\ell$-adic local systems over a curve defined over a finite field $\mathbb{F}_q$. The story starts from a two-page article by V. Drinfeld \cite{Drinfeld} who observed that the number of two-dimensional geometrically irreducible $\ell$-adic continuous representations of $\pi_1(X\otimes\closure{\mathbb{F}}_q)$ that can be extended to a representation of $\pi_1(X)$ (we ignore the base point for now) behaves as if it were expressed by a Lefschetz fixed-point formula on an algebraic variety over $\mathbb{F}_q$. Notably, these quantities are independent of $\ell$.

The case of Deligne's conjecture proved in the current article can be seen as a generalization of Arinkin's result which is written down by Deligne in \cite{Deligne} where the tame monodromy action around each puncture is ``split". Arinkin shows that, in this case, the number is equal to that of indecomposable parabolic bundles. After Jyun-Ao Lin's result \cite[Theorem 1.4]{Lin} or G. Dobrovolska, V. Ginzburg, and R. Trakvin's result \cite{DGT}, the latter number is further related to parabolic Higgs bundles. 

In addition to the case solved by D. Arinkin, there exist other instances where this conjecture is confirmed. When the ramifications are unipotent with only one Jordan block, and there are at least two such ramifications, the conjecture is verified by Deligne-Flicker \cite{Deligne-Flicker}. Flicker \cite{Flicker} has demonstrated the validity of the conjecture for the case of rank $2$ with one unipotent ramification. Drinfeld's original result, i.e., no ramifications, is extended to higher ranks in \cite{Yu1}.  
Recently, another paper by the author (\cite{Yu3}) addresses the rank $2$ case with all possible tame ramifications. 
However, we have yet to develop a theory that explains these results and leads to a comprehensive understanding of them. We refer to \cite{IHES} and \cite{Deligne} for more discussions.

\subsection{Counting $\ell$-adic local systems}\label{11C}
In order to present our results, we introduce the tame regular semisimple and generic local monodromies in the $\ell$-adic setting. Additionally, we introduce the concept of logarithmic Higgs bundles with prescribed residues.

Let $X$ be a smooth, projective, and geometrically connected curve defined over $\mathbb{F}_q$. Let $\closure{\mathbb{F}}_q$ be an algebraic closure of $\mathbb{F}_q$, and let $\closure{X} = X \otimes \closure{\mathbb{F}}_q$. Suppose ${S}$ is a finite set of closed points in ${X}$, and let $\closure{S}=S\otimes \closure{\mathbb{F}}_q$. We can identify $\closure{S}$ with a subset of $\closure{X}(\closure{\mathbb{F}}_q)$.

For each $x\in \closure{S}$, we choose a rank $n$ tame $\ell$-adic local system $\mathfrak{R}_x$ over the punctured disc $\closure{X}_{(x)}^{\ast}= \closure{X}_{(x)}-\{x\}$, where $\closure{X}_{(x)}$ is the henselisation of $\closure{X}$ at $x$. We denote by $\mathfrak{R}:= (\mathfrak{R}_x)_{x\in \closure{S}}$ the family of these local systems.
 We define \[E_n(\mathfrak{R})\] to be the set of isomorphism classes of $\ell$-adic local systems $\mathfrak{L}$ on $\closure{X}-\closure{S}$ of rank $n$ such that the inverse image $\mathfrak{L}|_{\closure{X}_{(x)}^{\ast}}$ of $\mathfrak{L}$ is isomorphic to $\mathfrak{R}_x$ for all $x\in \closure{S}$. Let \[E_n^{irr}(\mathfrak{R})\subseteq E_n(\mathfrak{R})\]
be the subset consisting of irreducible $\ell$-adic local systems. Suppose that the family $\mathfrak{R}$ is fixed by the pullback of the Frobenius endomorphism, i.e., for any $x\in \closure{S}$
\begin{equation}\label{13}\Frob^{*}\mathfrak{R}_{\Frob(x)}\cong \mathfrak{R}_x.\end{equation}
Then $\Frob^{*}$ induces a permutation on $E_n^{irr}(\mathfrak{R})$.

Let $\closure{\eta}$ be a geometric point lying over the generic point of $\closure{X}$ and $\closure{X}_{(x)}^{\ast}$. Then the tame fundamental group \(\pi_1(\closure{X}_{(x)}^{\ast}, \closure{\eta})^{t}\) is isomorphic to \[\hat{\mathbb{Z}}^{p'}(1):= \lim_{\underset{d}{\leftarrow}}\mu_d(\closure{\mathbb{F}}_q),  \]
where the projective limit is taken over the sets of $d$-th roots of unity of $\closure{\mathbb{F}}_q$ ($(d,q)=1$), ordered by inclusion and the transition function from $\mu_m(\closure{\mathbb{F}}_q)$ to $\mu_d(\closure{\mathbb{F}}_q)$ is given by $x\mapsto x^{m/d}$. 
For each $x\in \closure{S}$, the tame local system $\mathfrak{R}_x$ corresponds to an $\ell$-adic representation of $\hat{\mathbb{Z}}^{p'}(1)$ of rank $n$.  The tame local system $\mathfrak{R}_x$ is regular semisimple if and only if 
\[ \mathfrak{R}_x= \bigoplus_{i=1}^n \varepsilon_{x, i} \]
where $\varepsilon_{x, i}$ are pairwise non-isomorphic characters of $\hat{\mathbb{Z}}^{p'}(1)$. A necessary condition for $E_n(\mathfrak{R})$ to be non-empty (\cite[2.10]{Deligne}) is that 
\begin{equation}\label{prode=1}\prod_{x\in \closure{S} }\prod_{i=1}^n\varepsilon_{x, i}=1,  \end{equation}
as a character of $\hat{\mathbb{Z}}^{p'}(1)$.
We say that $\mathfrak{R}$ is generic if
\begin{equation}\label{generic} E_n^{irr}(\mathfrak{R})= E_n(\mathfrak{R}). \end{equation} 
Assuming \eqref{prode=1}, Deligne provides a condition (\cite[(2.10.3)]{Deligne}) that is sufficient for $\mathfrak{R}$ to be generic. By applying this criterion, we can observe that most cases meet the generic condition, which is why they are referred to as ``generic".

Let $K_X$ be a canonical divisor and \[ D=K_X+\sum_{v\in S}v.\] Suppose $e$ is an integer coprime to $n$. 
Let $\mathfrak{M}_n^{e}(D)$ be the coarse moduli space of semistable logarithmic Higgs bundles of rank $n$ and degree $e$ with simple poles in $S$. Recall that a logarithmic Higgs bundle is a pair $(\mathcal{E}, \varphi)$ with a vector bundle $\mathcal{E}$ over $X$ and a bundle morphism: \[\varphi: \mathcal{E}\longrightarrow \mathcal{E}\otimes_{\mathcal{O}_X} \mathcal{O}_X(D)\]
A pair $(\mathcal{E}, \varphi)$ is semistable if every non-zero sub-Higgs bundle has a slope ($=\deg/\mathrm{rank}$) smaller than or equal to $e/n$. 
\begin{remark*}
In the literature, the notation $\Omega_{X/\mathbb{F}_q}^{1}(log \sum_{v\in S}v)$ is often used to denote $\mathcal{O}_X(D)$. 
%\[ \mathcal{E}\otimes , \] here \[ \mathcal{E}\otimes \Omega_{X/\mathbb{F}_q}^{1}(log \sum_{v\in S}v)\cong \mathcal{E}\otimes_{\mathcal{O}_X} \mathcal{O}_X(D). \]
\end{remark*}

For each $v\in S$, let \[\mathcal{R}_v:=\{z^n+a_1 z^{n-1}+\cdots+a_n\in \kappa_v[z] \}\] where $\kappa_v$ is the residue field of $v$. The space $\mathcal{R}_v$ is a $\kappa_v$-linear space. It can be viewed as an $\mathbb{F}_q$-linear space by forgetting $\kappa_v$-structure (i.e., it is isomorphic to $R_{\kappa_v/\mathbb{F}_q}V$, the restriction of scalars from $\kappa_v$ to $\mathbb{F}_q$ of a $\kappa_v$-vector scheme $V$ of dimension $n$). 
We have an $\mathbb{F}_q$-morphism of schemes (see the paragrah before Remark \ref{RMKS3} for its definition) : 
\[ \res: \mathfrak{M}_n^{e}(D)\longrightarrow \prod_{v\in S}\mathcal{R}_v.  \]
It sends a pair $(\mathcal{E}, \varphi)$ to its characteristic polynomial of the fiber map $\varphi_v$ of $\varphi$ in $v$. For each $o\in \prod_{v\in S}\mathcal{R}_v(\mathbb{F}_q)$, we denote by $\mathfrak{M}_n^{e}(o)$ the fiber $\res^{-1}(o)$. 

\begin{adefinition}
Let $o=(o_v)_{v\in S}$ be a point in $\prod_{v\in S}\mathcal{R}_v(\mathbb{F}_q).$ For each $v\in S$, $o_v$ is a monic polynomial of degree $n$ with coefficients in $\kappa_v$.
We say that such a point $o$ is similar to $\mathfrak{R}$ if the following conditions hold. 
\begin{enumerate}
\item Each $o_v$ is regular in the sense that all of its roots are distinct (such a polynomial is often called separable, but a separable polynomial must have two by two distinct roots does not seem to be universally accepted). 
\item Let $a_v\in \kappa_v$ be the sum of roots of $o_v$. 
The point $o$ satisfies the following necessary condition for $\mathfrak{M}_n^e(o)$ to be non-empty: 
\begin{equation}\sum_{v\in S} \Tr_{\kappa_v/\mathbb{F}_q}(a_v)=0 . \end{equation}
\item
Let $v\in S$ and $x\in \closure{S}$ a point lying over $v$.  Suppose that $v$ has degree $d_v$, i.e. $\{v\}\otimes\closure{\mathbb{F}}_q$ has $d_v$ points. 
Since $\Frob^{ d_v}(x)=x$, 
the condition \eqref{13} implies that $\Frob^{\ast d_v}$ permutes the set of irreducible factors of $\mathfrak{R}_x$. Its orbits define a partition of $n$. This partition depends only on $v$, not on $x$. The degrees of the irreducible factors of $o_v$ also define a partition of $n$. Assume that these two partitions coincide. 
\end{enumerate}
\end{adefinition}

One may think of $o$ as ``$\log \mathfrak{R}$" if it is similar to $\mathfrak{R}$.  
We show in Appendix \ref{regulartra} that if $p\neq 2$ and $q\geq n(n-1)/2$, then for any $S$ and $(\mathfrak{R})_{x\in \closure{S}}$ there exists an $o$ that is similar to $\mathfrak{R}$. I thank a referee who points out to me that such an element does not always exist.

\begin{atheorem}
\label{THEO}
Let  $\mathfrak{R}=(\mathfrak{R}_x)_{x\in\closure{S}}$ be a family of rank $n$ tame regular semisimple generic (\eqref{generic}) $\ell$-adic local systems over $(\closure{X}_x^{\ast})_{x\in\closure{S}}$ satisfying \eqref{prode=1}. Suppose that $\mathfrak{R}$ is fixed by Frobenius in the sense of \eqref{13} so that $\Frob$ acts on $E_n^{irr}(\mathfrak{R})$ via pullback. We define $E_n^{irr}(\mathfrak{R})^{\Frob^{\ast k}}$ to be the set of points fixed by the $k$-iterated action of $\Frob^{\ast}$. Let \[ o \in \prod_{v\in S}\mathcal{R}_v(\mathbb{F}_q)\] be any point similar to $\mathfrak{R}$.
Let $e$ be an integer such that $(e, n)=1$. Then we have \[|E_n^{irr}(\mathfrak{R})^{\Frob^{*k}}| = q^{-\frac{k}{2}\dim\mathfrak{M}_n^{e}(o) }| \mathfrak{M}_n^{e}(o)(\mathbb{F}_{q^k}) |, \quad \forall k\geq 1.   \]

\end{atheorem}

It is a corollary of the theorem that the number $|\mathfrak{M}_n^{e}(o)(\mathbb{F}_{q^k}) |$ ($k\geq 1$) does not depend on the choice of $o$ that is similar to $\mathfrak{R}$. Therefore, for any $S$,  $|E_n^{irr}(\mathfrak{R})^{\Frob^{*k}}|$ depends only on the similarity class of $ \mathfrak{R}$ at least when such an $o$ exists.  

 If $k$ divisible enough so that $S$ and $o$ splits over $\mathbb{F}_{q^k}$, then $|\mathfrak{M}_n^{e}(o)(\mathbb{F}_{q^k})|$ is explicit by a result of A. Mellit (see Theorem \ref{AL} and comments below that theorem). His result suggests that $|E_n^{irr}(\mathfrak{R})^{\Frob^{*k}}|$ depends only on the Frobenius action on $(\closure{S}, (\mathfrak{R}_{x})_{x\in\closure{S}})$, but not on the inclusion $S\subseteq X$, nor the explicit description of $\mathfrak{R}$. In other words, we expect that the counting depends only on the isomorphism class of $(\closure{S}, (\mathfrak{R}_{x})_{x\in\closure{S}})$ with $\Frob$-action.

When $g=0$ and $\deg S\leq 2$, we have $\mathfrak{M}_n^{e}(o)=\emptyset$. Otherwise, we have  (see the discussion after \eqref{525}) \[ \dim  \mathfrak{M}_n^{e}(o) =2((g-1)n^2+1+\deg S\frac{n^2-n}{2}).\] 
Theorem \ref{THEO} proves Conjecture 2.15 (i) and (iii) of \cite{Deligne} in the generic case. 
Based on Theorem \ref{THEO}, we prove the following result which confirms Conjecture 6.3 of \cite{Deligne}. 

A $q$-Weil integer is an algebraic integer $\alpha$ such that for any field embedding $\mathbb{Q}(\alpha)$ in $\mathbb{C}$, the absolute value of $\alpha$ in $\mathbb{C}$ is $q^{m/2}$ for some $m\in \mathbb{N}$ independent of the embedding. 
\begin{atheorem}\label{malpha}
Under the condition of Theorem \ref{THEO},
there are $q$-Weil integers $\alpha$ and integers $m_\alpha$ such that 
\[ |E_n^{irr}(\mathfrak{R})^{\Frob^{*k}}| =|\Pic^{0}_X(\mathbb{F}_{q^k})|\sum_{\alpha} m_\alpha \alpha^k, \quad \forall k\geq 1.   \]
The $\alpha$ are monomials in $q$-Weil numbers of the curve $X$ and roots of unity. Here $q$-Weil numbers of the curve $X$ refer to the eigenvalues of the Frobenius action on $\ell$-adic cohomology of $\closure{X}$. 
\end{atheorem}

%Thanks to the work of Mellit \cite{Mellit}on the counting of parabolic Higgs bundles, we can obtain an explicit formula in some cases (cases where $S\subseteq X(\mathbb{F}_q)$ and every $T_v$ is split in Theorem \ref{THEO}). 

\subsection{The split case: revisit Arinkin's case}
Suppose that $S\subseteq X(\mathbb{F}_q)$, i.e. the degree of each point $x$ in $S$ is $1$.  
Let $\xi=(\xi_x)_{x\in S}\in (\mathbb{Q}^n)^S$ be a set of parabolic weights. Assume it is admissible and generic (see Definition \ref{adgen} and Definition \ref{admissible}). Let $\mathcal{M}_{n, S}^{e,\xi}$ be the moduli space of stable parabolic Higgs bundles over $X$ constructed by Yokogawa (see Section \ref{parabolic} for definitions). 
 In the literature, the parabolic Higgs bundles considered in the current article are sometimes also called strictly parabolic.

The following result (Theorem \ref{AL}) can be obtained by combining Arinkin's result, which expresses $|E_n^{irr}(\mathfrak{R})^{\Frob^{*k}}|$ in terms of the number of parabolic indecomposable vector bundles and J.-A. Lin's equality \cite[Theorem 1.4]{Lin} between the latter and the number of parabolic Higgs bundles up to a power of $q$. Lin's equality is also done in greater generality by Dobrovolska, G., V. Ginzburg, and R. Trakvin in \cite{DGT}.  We give a different and direct proof.

In fact, the case treated by Arinkin is a bit more general which allows non-regular local monodromies as well. He assumes that each $\mathfrak{R}_x$ is semisimple, and its irreducible factors are fixed by Frobenius endomorphism. In the terminology of representations, it means that $\mathfrak{R}_x$ corresponds to a semisimple representation of $\hat{\mathbb{Z}}^{p'}(1)$ that factors through a quotient map $\hat{\mathbb{Z}}^{p'}(1)\rightarrow \kappa_v^{\times}$. 
 For more details, see Théorème 3.5 of \cite{Deligne}. 
Our regular assumption is that $\mathfrak{R}_x$ is a direct sum of pairwise different irreducible factors, while we allow Frobenius endomorphism permutes the factors of $\mathfrak{R}_x$. The representation of $\hat{\mathbb{Z}}^{p'}(1)$ corresponding to $\mathfrak{R}_x$ factors through 
the quotient map \( \hat{\mathbb{Z}}^{p'}(1) \rightarrow \kappa^\times \) for a degree $n$ extension $\kappa$  of $\kappa_v$.

\begin{atheorem}[Arinkin-Lin] \label{AL}
Suppose that $S\subseteq X(\mathbb{F}_q)$ is a finite set of $\mathbb{F}_q$-rational points. Suppose that for each $x\in {S}$, $\Frob^{*}$ fixes every irreducible factor of $\mathfrak{R}_x$.
 Let $(e,\xi)\in \mathbb{Z}\times (\mathbb{Q}^n)^{S}$ be admissible and generic (see Definition \ref{adgen} and Definition \ref{admissible}).
 Then \begin{equation}\label{EAL}
 |E_n^{irr}(\mathfrak{R})^{\Frob^{*k}}| = q^{-\frac{k}{2} \dim \mathcal{M}_{n, S}^{e,  {\xi}} }|\mathcal{M}_{ n, S}^{e,  {\xi}}(\mathbb{F}_{q^k})|, \quad \forall k\geq 1. \end{equation}
\end{atheorem}

Theorem \ref{AL} and Theorem \ref{THEO} implies that if each $o_v$ splits into linear polynomials,  then $|\mathfrak{M}_n^{e'}(o)(\mathbb{F}_{q^k})|= |\mathcal{M}_{ n, S}^{e,  {\xi}}(\mathbb{F}_{q^k})|$ for all $k\geq 1$, all $e'$ coprime to $n$ and all dmissible and generic pair $(e,  {\xi})$.

Mellit \cite[Corollary 7.9]{Mellit} has provided an explicit formula for the quantity on the right-hand side of the equation \eqref{EAL}.  His result enables us to get more information on the numbers $\alpha$ and $m_\alpha$  in Theorem \ref{malpha}.

The moduli space $\mathcal{M}_{n, S}^{e,  {\xi}} $ admits a $\mathbb{G}_m$ action via dilation on the Higgs field. Let $(\mathcal{M}_{n, S}^{e,  {\xi}})^{\mathbb{G}_m}$ be the fixed point variety. 
We then prove the following result. 
\begin{atheorem}\label{1.4}
Under the condition of Theorem \ref{AL}, for any $k\geq 1$, we have
\begin{equation}\label{OKAY}
|E_n^{irr}(\mathfrak{R})^{\Frob^{*k}}| =|(\mathcal{M}_{n, S}^{e,  {\xi}})^{\mathbb{G}_m}(\mathbb{F}_{q^k})|. \end{equation}
\end{atheorem}
A natural question arising from this result is whether a natural bijection exists between the two sets in the equality \eqref{OKAY}. However, working directly with the viewpoint of $\ell$-adic local systems or $\ell$-adic representations can be challenging. An alternative approach is to construct motivic local systems or to use Abe's theorem \cite{Abe} to move to the $p$-adic setting. Recently, Yang and Zuo \cite{YZ} have tackled this problem for rank $2$ $\ell$-adic local systems over $X-S=\mathbb{P}^1-\{0,1,\infty, \lambda\}$ with some prescribed tame generic and quasi-unipotent ramifications. We also draw the reader's attention to Deligne's proposals in \cite{IHES} and \cite{Deligne}.

\subsection{Main ingredients of the proof}
We provide a brief overview of the proof which employs methods from the theory of automorphic forms.  

It is sufficient to consider the case $k=1$ in Theorem \ref{THEO} and Theorem \ref{1.4}, the general case is deduced by base change (arguments for this are provided in the begining of Section \ref{Compar}). 

Starting with the Langlands correspondence established by Drinfeld and Lafforgue, and C. Bushnell and G. Henniart's result on the explicit local Langlands correspondence, we reduce the question to a counting problem of certain cuspidal automorphic representations of $GL_n(\AAA)$, where $\mathbb{A}$ is the ring of adèles for the function field $F=\mathbb{F}_q(X)$. To be a bit more precise, Bushnell-Henniart's result enables us to construct an irreducible representation $\rho$ of $GL_n(\mathcal{O})$ (Proposition \ref{623}), constructed via Deligne-Lusztig theory, so $|E_n^{irr}(\mathfrak{R})^{\Frob^\ast}|$ equals the number of inertial equivalence classes of automorphic cuspidal representations $\pi$ of $GL_n(\AAA)$ that contain $\rho$: $\pi_\rho\neq 0$. Here $\mathcal{O}$ is the ring of integral adèles (see Theoerm \ref{fixat}).

To address this counting problem on the automorphic side, we use the Arthur-Selberg trace formula, which has been established over a function field by L. Lafforgue.
We use the character (i.e. the trace function) of the contragredient of $\rho$ as the test function. 
Under the generic condition, the spectral side of the trace formula is very simple. It simply equals the counting (Theorem \ref{Einfinitesimal}). 

For the geometric side of the trace formula, we exploit the fact that the test function has small support (is contained in $GL_n(\mathcal{O})$). We use a coarse expansion following the characteristic polynomials of elements in $GL_n(F)$. 
Inspired by a result of Chaudouard, we find that in the generic case, most terms in the coarse expansion of the geometric side of the trace formula vanish, except for essentially unipotent ones (Theorem \ref{vanishing}). We hence reduce to treat the unipotent part of the trace formula.

The main idea is to use a trace formula for Lie algebra.  After introducing the Hitchin moduli space in Section \ref{B}, we finish the proof in Section \ref{F5}.
We identify the unipotent term of the Arthur-Selberg-Lafforgue trace formula with a nilpotent term in a trace formula for Lie algebra. Through a Fourier transform and the use of Kazhdan-Springer's hypothesis, it is further related to a Lie algebra trace formula for a relatively simple test function. Finally, we interpret this latter Lie algebra trace formula as a point counting of logarithmic Higgs bundles with fixed residues, utilizing Weil's dictionary between adèles and vector bundles (Theorem \ref{B3}). Throughout the proof, we observe that Arthur's truncation is the same as semistability on the Higgs bundles, which can be demonstrated through combinatoric arguments.

We need to introduce the moduli space of parabolic Higgs bundles to state Theorem \ref{AL}, and we also need Theorem \ref{PGm} about it to prove Theorem \ref{malpha}. 
In order to establish a connection with parabolic Higgs bundles, we must first modify the trace formula in Arthur's construction by introducing an additional parameter $\xi$ that serve as parabolic weights in the construction of the moduli space of parabolic Higgs bundles. This modification has been carried out in \cite{Yu2}. We then demonstrate that, in the generic case, the coarse expansion introduced in \textit{op. cit.} still yields a similar vanishing result (Theorem \ref{vanishing}). %The remaining arguments follow a similar path as in the previous case, and we once again utilize Weil's dictionary to establish a connection with parabolic Higgs bundles. We use some combinatorial arguments to study the stability of parabolic Higgs bundles and the modified truncation, which are based on the use of Behrend's complementary polyhedra.

The proof of Theorem \ref{malpha} needs Theorem \ref{THEO}, Theorem \ref{AL}, and Theorem \ref{1.4}. It is completed in Section \ref{finalsection}.

Although the main tool we use in this article is Arthur's trace formula as in \cite{Yu1} where the main difficulty lies in the spectral side of the trace formula, the focus here is quite different. In this paper, the genericity condition allows us to completely reduce the difficulty from the spectral side of the trace formula at the cost of increasing the complexity of the geometric side. 

\subsection{Acknowledgements}
I thank Erez Lapid for the discussions and for his valuable suggestions. I thank Peiyi Cui for providing me with a reference. I thank several anonymous referees, whose comments helped me to improve this manuscript.  
This article is written while the author was at the Weizmann Institute of Science and was supported by the grant BSF 2019274.

\section{Global and local Langlands correspondences}
We continue to use the notation in Section \ref{11C}. 
In this section, we reduce the calculation of the cardinality of $E_n^{irr}(\mathfrak{R})^{\Frob^{*}}$ to a question of counting certain automorphic representations of $GL_n$ through the global Langlands correspondence established by Lafforgue \cite{Lafforgue}. The local monodromies can also be transferred to the automorphic side, described by the local Langlands correspondence established by Laumon, Rapoport, and Stuhler \cite{LRS}. The tame case, the only case we are concerned with, has been made explicit by Bushnell and Henniart in their works \cite{BH} and \cite{BH3}.

\subsection{General notation}
We define $G$ to be the group scheme $GL_n$ over $\mathbb{Z}$.  Let $F=\mathbb{F}_q(X)$ denote the function field of $X$. Let $|X|$ be the set of closed points of $X$. We identify $|X|$ with the set of places of $F$. For each place $v\in |X|$, let $F_v$ be the local field of $F$ in $v$, $\mathcal{O}_v$ the integer ring of $F_v$ and $\kappa_v$ the residue field of $\mathcal{O}_v$.
We fix a uniformizer $\wp_v$ for the maximal idea of $\mathcal{O}_v$. 
For each $m\geq 1$, $\mathbb{F}_{q^m}$ denotes the sub-field with $q^m$ elements in $\closure{\mathbb{F}}_q$. Let $\mathbb{A}$ be the ring of adèles and $\mathcal{O}=\prod_v\mathcal{O}_v$. The degree map of $\mathbb{A}^\times $ is defined by
\[\deg ((x_v)_{v\in |X|}) = -\sum v(x_v)[\kappa_v:\mathbb{F}_q], \]
the valuations $v$ are normalized to be surjective to $\mathbb{Z}$.

An $\ell$-adic representation is a continuous $\closure{\mathbb{Q}}_{\ell}$-linear representation of finite dimension. We fix an isomorphism: \[ \iota: \closure{\mathbb{Q}}_{\ell}\overset{\sim}{\longrightarrow} \mathbb{C}. \]

\subsection{The set $E_n^{irr}(\mathfrak{R})^{\Frob^\ast}$ in terms Weil group representations}
\label{S2.2}

Below is a brief summary of the material in sections 2.7-2.9 of \cite{Deligne}. We refer to \cite[Section 2.7-2.9]{Deligne} and \cite[Section 2.1]{Yu1} for more details. 

Let $F_1=F\otimes_{\mathbb{F}_q} \closure{\mathbb{F}}_q$ be a field extension of $F$. 
Let $\closure{F}$ be an algebraic closure of $F_1$. The algebraic closure $\closure{F}$ of $F$ defines a geometric point $\closure{\eta}=\Spec(\overline{F})$ of $\closure{X}$ that lies over the generic point of $\closure{X}$.

%Let $\mathcal{O}_{(v)}$ denote the Henselization of the local ring $\mathcal{O}_{X,v}$ and let $F_{(v)}$ be its fraction field. 
Let $F_{1, (x)}$ be the function field of  $\closure{X}_{(x)}^{\ast}$, so we have
\[ \closure{X}_{(x)}^{\ast}=\Spec(F_{1, (x)}) . \]
Fix a point $x_0$ lying over $v\in S$, we fix an embedding $F_{1, (x_0)}$ in $\closure{F}$ which defines a geometric point $\closure{\eta}=\Spec(\closure{F})$ of $\closure{X}_{(x_0)}^{\ast}$. Since the points of $\closure{X}$ that lie over $v$ form a $\Frob$-orbit, the Frobenius endomorphism induces a geometric point of $\closure{X}_{(x)}^{\ast}$ for every $x$ lying over $v$. 
The tame fundamental group $\pi_1(\closure{X}_{(x)}^{\ast}, \closure{\eta})^{t}$ is isomorphic to 
\( \hat{\mathbb{Z}}^{p'}(1)\). %:=\lim \mu_{n}(\closure{\mathbb{F}}_q).  \) The projective limit is taken over the set of subgroups $\mu_n(\closure{\mathbb{F}}_q)$ of $\closure{\mathbb{F}}_q^{\times}$ ($n$ prime to $p$), ordered by inclusion and the transition morphism from $\mu_{n}(\closure{\mathbb{F}}_q)$ to $\mu_{n}(\closure{\mathbb{F}}_q)$ is the map $x\mapsto x^{m/n}$.
 The Frobenius endomorphism: $\Frob: \closure{X}_{(x)}^{\ast}\rightarrow \closure{X}_{(\Frob(x))}^{\ast}$ induces a morphism between tame fundamental groups that is multiplication by $q$ map on $ \hat{\mathbb{Z}}^{p'}(1)$.

 %Let $F_{v}$ be the $v$-adic completion of $F_{(v)}$. 

The tame family $(\mathfrak{R}_x)_{x\mid v}$ over $(\closure{X}_{(x)}^{\ast})_{x\mid v}$ is equivalent to a family of $\ell$-adic representations $(\sigma_x)_{x\mid v}$ of rank $n$ of $\hat{\mathbb{Z}}^{p'}(1)$. 
 Let $\closure{F}_{v}$ be an algebraic closure of $F_v$ and embed $\closure{F}$ into $\closure{F}_v$.  
Let $W_F$ be the Weil group of $F$, and $W_{F_v}$ be the local Weil group of $F_v$ and $I_{F_v}$ the inertial subgroup of $F_v$. We have an embedding $W_{F_v}\rightarrow W_F$ given by the inclusion $\closure{F}\subseteq \closure{F}_v$.   
We say that two $\ell$-adic representations of $W_{F_v}$ are inertially equivalent if their restrictions to $I_{F_v}$ are isomorphic. 
The hypothesis that the isomorphism class of $(\mathfrak{R}_{x})_{x\in \closure{S}}$ is fixed  by Frobenius endomorphism:
\[\Frob^{\ast}\mathfrak{R}_{\Frob(x)}\cong \mathfrak{R}_{x}, \quad \forall x\in \closure{S}, \]
implies that $(\sigma_x)_{x\mid v}$ defines an inertial equivalence class $\mathfrak{R}_v$ of $\ell$-adic representations of rank $n$ of $W_{F_v}$ for every $v\in S$. In fact, the closure $F_{1, x_0}$ of $F_{1, (x_0)}$ in $\closure{F}_v$ is the maximal unramified extension of $F_{v}$ inside $\closure{F}_v$, we have
\[  I_{F_v}\cong \pi_{1}(\closure{X}_{(x_0)}^{\ast}, \closure{\eta}), \]
and the inertial class is determined by  \( \sigma_{x_0}\), which can be extended to a representation of $W_{F_v}$.

We define two $\ell$-adic representations $\sigma_1$ and $\sigma_2$ of $W_F$ to be inertially equivalent if their restrictions to $W_F^{0}$, the degree $0$ subgroup of $W_F$, are isomorphic. If $\sigma_1|_{W_F^0}$ is irreducible, then $\sigma_1$ is inertially equivalent to $\sigma_2$ if and only if there exists a character $\chi$ of $W_F$ that factors through the degree map, such that \[ \sigma_1\cong \sigma_2\otimes \chi. \]
\begin{prop}\label{weilrep} 
The set $E_n^{irr}(\mathfrak{R})^{\Frob^\ast}$ is in bijection with the set of inertial equivalence classes of $\ell$-adic irreducible representations of rank $n$ of $W_F$, satisfying the following conditions: their restrictions to $W_{F_v}$ belong to the inertial equivalence class $\mathfrak{R}_v$ for every $v\in S$, and is trivial on $I_{F_v}$ for every $v\notin S$. \end{prop}
\begin{proof}[Sketch of the proof]
An $\ell$-adic local system over an open set of $\closure{X}$ is equivalent to an $\ell$-adic representation of $W_F^{0}$. The isomorphism class of the local system is fixed by the Frobenius endomorphism if and only if the representation can be extended to a representation of $W_F$ (see \cite[Lemma 2.1.2]{Yu1}). The $\ell$-adic sheaf being smooth $\overline{X}-\overline{S}$ is equivalent to the representation being unramified for all $v\notin S$, and the ramification $\mathfrak{R}$ of the local system is translated into the prescribed ramifications $\mathfrak{R}_v$ for $v\in S$. Finally, note that we assume $\mathfrak{R}$ is generic. It guarantees that the restriction to $W_F^{0}$ of such a representation is automatically irreducible since $E_n^{irr}(\mathfrak{R})= E_n(\mathfrak{R})$.
\end{proof}

%We fix a topological generator in $\hat{\mathbb{Z}}^{p'}(1)$, denote it by $``1"$:  \[ 1\in \hat{\mathbb{Z}}^{p'}(1).\] The family $(\sigma_x)_{x\mid v}$ is determined by the image of $1$ via $\sigma_x$. Denote the family by $(c_x)_{x\mid v}$ with $c_{x}\in GL_n(\closure{\mathbb{Q}}_{\ell})$. 

In the next section, we explain a result of Bushnell and Henniart to specify $\mathcal{R}_v$.

\subsection{Explicit tame local Langlands correspondence for \(GL_n\)}\label{Expli}

Let \[\mathcal{A}^{0}_n(F_v)_0 \] be the set of isomorphism classes of supercuspidal representations of $G(F_v)$ that are of depth zero (i.e., admitting a non-zero vector fixed by $1+\wp_v M_n(\mathcal{O}_v)$). Let \[\mathcal{G}^{0}_n(F_v)_0\] be the set of isomorphism classes of irreducible tamely ramified $n$-dimensional representations of $W_{F_v}$ (i.e., representations that are trivial on the wild inertial subgroup of $I_{F_v}$).

A pair $(E/F_v, \chi)$ consisting of a separable field extension $E/F_v$ and a (complex) character $\chi$ of $E^{\times}$ is called an admissible (depth zero) pair if $E/F_v$ is unramified, $\chi|_{1+\wp_E}$ is trivial and the conjugates $\chi^{\sigma}$, $\sigma\in \Gal(E/F_v)$ are pairwise distinct. 
The set of isomorphism classes of admissible pairs such that $[E:F_v]=n$ is denoted by  \[P_n(F_v)_0. \]
We will establish two bijections $P_n(F_v)_0\rightarrow \mathcal{A}^{0}_n(F_v)_0$ and $P_n(F_v)_0\rightarrow \mathcal{G}^{0}_n(F_v)_0$. By using these bijections, we will express the local Langlands correspondence between \(\mathcal{A}^{0}_n(F_v)_0\) and  \(\mathcal{G}^{0}_n(F_v)_0\).  %(which is not their composition). 

We establish the bijection:
\begin{align*}
P_n(F_v)_0&\longrightarrow \mathcal{A}^{0}_n(F_v)_0, \\
(E/F_v,\chi)&\longmapsto \pi_\chi.
\end{align*}
 Bushnell and Henniart give this bijection in \cite[2.2]{BH}, where they use Green's parametrization instead of Deligne-Lusztig's construction. However, both constructions lead to the same result, as shown in \cite[2.1]{BH3} and \cite[4.2]{DL}.
Let $(E/F_v, \chi)\in P_n(F_v)_0$. Then $\chi|_{\mathcal{O}_E^{\times}}$ is the inflation of a character $\closure{\chi}$ of $\kappa_E^{\times}$.  
Let $U={R}_{\kappa_E/\kappa_v}\mathbb{G}_{m}$, the restriction of scalars of the torus $\mathbb{G}_{m}$ from $\kappa_{E}$ to $\kappa_v$. 
There is a unique conjugacy class of tori of  $G$ defined over $\kappa_v$ that is isomorphic to $U$. We can thus view $U$ as a torus of $G$, and we view $\closure{\chi}$ as a character of $U(\kappa_v)$. Then, it is in general position. The Deligne-Lusztig induced representation \[ (-1)^{n-1}\iota(R_{U(\kappa_v)}^{G(\kappa_v)}(\iota^{-1}\circ\closure{ \chi}))\]  is then irreducible and cuspidal. By inflation, we obtain a representation $\lambda_\chi$ of $G(\mathcal{O}_v)$. We extend $\lambda_\chi$ to a representation $\Lambda_\chi$ of $F_v^{\times}G(\mathcal{O}_v)$ by demanding that $\Lambda_\chi|_{F_v^{\times}}$ is a multiple of $\chi|_{F_v^\times}$. We set \[\pi_\chi=\mathrm{c\text{-}Ind}_{F_v^{\times}G(\mathcal{O}_{v})}^{G(F_v)}\Lambda_\chi,\]
then $\pi_\chi\in \mathcal{A}^{0}_n(F_v)_0$. And the map $(E/F_v,\chi)\mapsto \pi_\chi$ establishes a bijection between $P_n(F_v)_0$ and $\mathcal{A}^{0}_n(F_v)_0$.

The bijection \begin{align*}
P_n(F_v)_0&\longrightarrow \mathcal{G}^{0}_n(F_v)_0, \\
(E/F_v,\chi)&\longmapsto \sigma_\chi,
\end{align*}
is easy to define. Let $(E/F_v, \chi)\in P_n(F_v)_0$, it suffices to define $\sigma_\chi$ by \[\sigma_\chi=\Ind_{W_E}^{W_{F_v}}\chi,\]
where $\chi$ is viewed as a character of $W_E$ by local class field theory.  See \cite[A.3]{BH} for the fact that this is well-defined and is bijective.

The following result is obtained in Theorem 2 in Section 2.4 and Corollary 3 of Section 5.2 of \cite{BH3}. 
\begin{theorem}[Bushnell-Henniart]\label{BH}
The local Langlands correspondence induces a bijection between  \(\mathcal{G}^{0}_n(F_v)_0\) and  \(\mathcal{A}^{0}_n(F_v)_0\). Given an admissible pair $(E/F_v, \chi)$, the local Langlands correspondence sends $\sigma_{\chi}$ to $\pi_{\mu_n\chi}$ where \[\mu_n: E^\times \xrightarrow{\deg}\mathbb{Z}\longrightarrow \{\pm 1\}\]
is the unramified character of $E^\times$ that sends a uniformizer to $(-1)^{n-1}$. 
\end{theorem}

The local Langlands correspondence is a one-to-one correspondence between $n$-dimensional  complex semi-simple Deligne representations of $W_{F_v}$ and complex smooth irreducible representations of $G(F_v)$. The category of $\ell$-adic (continuous) Frobenius-semisimple representations of $W_{F_v}$ is equivalent to the category of complex semi-simple Deligne representations (via $\iota$). 
A detailed presentation of the extension of Langlands correspondence from the irreducible-supercuspidal case can be found in \cite[4.4]{Rodier}. In our specific case, i.e. $\mathfrak{R}_x$ are tame regular semi-simple, the correspondence is straightforward as the nilpotent operator in a Deligne representation is trivial, and we need the following result.

Let $v\in S$. We have fixed a point $x_0\in \closure{S}$ be the point lying over $v\in S$ at the beginning of Section \ref{S2.2}.
The $\ell$-adic representation $\sigma_{x_0}$ of $\hat{\mathbb{Z}}^{p'}(1)$ corresponding to $\mathfrak{R}_{x_0}$ is a direct sum of pairwise different characters of $\hat{\mathbb{Z}}^{p'}(1)$. Let $|\kappa_v|=q_v=q^{d_v}$. Since $\sigma_{x_0}$ is fixed by $\Frob^{\ast d_v}$, we can assume that  
\[\sigma_{x_0}=\bigoplus_{i=1}^{r} \bigoplus_{a=1}^{n_i}\chi_i^{q_v^a}, \]
for a partition $(n_1, \ldots, n_r)$ of $n$, where $n_i$ is the smallest positive integer such that $\chi_{i}^{q_v^{n_i}}=\chi_i$. Each $\chi_i$ factors through the quotient map \[ \hat{\mathbb{Z}}^{p'}(1)\longrightarrow \mathbb{F}_{q_v^{n_i}}^\times, \] 
so it defines a character $\closure{\chi}_i$ of $\mathbb{F}_{q_v^{n_i}}^\times$. Let $U$ be a maximal torus defined over  $\kappa_v$ in $GL_n$ such that 
\[ U(\kappa_v)\cong \prod_{i=1}^{r} \mathbb{F}_{q_v^{n_i}}^\times.  \]
The character $\closure{\chi}=\closure{\chi}_{1}\otimes \cdots\otimes\closure{\chi}_{r}$ of $U(\kappa_v)$ is in general position and the Deligne-Lusztig induced character 
\[ \iota((-1)^{n-r} R_{U(\kappa_v)}^{G(\kappa_v)} \closure{\chi}) \]
is the character of an irreducible representation $\rho_v$ of $G(\kappa_v)$. We denote the inflated representation of $G(\mathcal{O}_v)$ still by $\rho_v$. It is clearly well-defined.

\begin{prop}\label{623}
An irreducible smooth representation $\pi$ of $G(F_v)$ is  associated to an $\ell$-adic representation in $\mathfrak{R}_v$ (via $\iota$) under local Langlands correspondence if and only if
 \[\pi_{\rho_v}\neq 0.  \]
 Here $\pi_{\rho_v}$ is the $\rho_v$-isotypic subspace of $\pi$. 
\end{prop}
\begin{proof}
Let $E_{i}$ be the unramified extension of $F_v$ inside $\closure{F}_v$ of degree $n_i$ and extend $\closure{\chi}_i$ to a character of $E_i^\times $, denoted by $\chi_i'$.   
Note that \[(\Ind_{W_{E_i}}^{W_{F_v}} \chi_i' )|_{I_{F_v}} \cong \bigoplus_{a=1}^{n_i} \chi_i^{q_v^{a}}. \]
We deduce that \[ \Sigma=\bigoplus_{i=1}^r \Ind_{W_{E_i}}^{W_{F_v}}\chi'_i\in \mathcal{R}_v. \]
Local Langlands correspondence (Theorem \ref{BH}) associates \[ \sigma_{\chi'_i}=\iota(\Ind_{W_{E_i}}^{W_{F_v}}\chi'_i)\] to the supercuspidal representation of $\pi_{\mu_{n_i}(\iota \circ\chi_i')}$ of $GL_{n_i}(F_v)$. 

Let $P$ be the standard parabolique subgroup of $G$ associated to the partition $(n_1, \ldots, n_r)$. Since $\pi_{\mu_{n_i}(\iota\circ\chi_i')}$ are pairwise non-inertially equivalent (because it is so in the Galois side), the normalized parabolic induction $\Pi$ of $(\pi_{\mu_{n_1}(\iota\circ\chi_1')}, \ldots, \pi_{\mu_{n_r}(\iota\circ\chi_r')})$ from $P(F_v)$ to $G(F_v)$ is irreducible. The local Langlands correspondence associates $\Sigma$ to $\Pi$. 
From the construction of $\Pi$, we see that 
\[ \Pi_{\rho_v}\neq 0. \]
On the other hand, if $ \Pi_{\rho_v}\neq 0$, then Moy and Prasad \cite[Proposition 6.6, Theorem 6.11(2)]{MP2} prove that the supercuspidal support of $\Pi$ is inertially equivalent to the representation $(\pi_{\mu_{n_1}(\iota\circ\chi_1')}, \ldots, \pi_{\mu_{n_r}(\iota\circ\chi_r')})$ of $M_P(F_v)$, where $M_P$ is the standard Levi subgroup of $P$. Hence the $\ell$-adic representation corresponding to $\Pi$ by the local Langlands correspondence belongs to the inertial equivalence class $\mathcal{R}_v$. \end{proof}

\subsection{Reduction of the problem by the global Langlands correspondence}
Throughout this article, automorphic representations are assumed to be irreducible. Two cuspidal automorphic representations $\pi$ and $\pi'$ are said to be inertially equivalent if there exists a character $\chi$ of $GL_n(\AAA)$ factoring through the composition of the degree and determinant maps, such that 
\[\pi\cong \pi'\otimes\chi.\] 
%The main objective of this section is to associate an irreducible representation $\rho$ of $GL_n(\mathcal{O})$ with the given ramification type $\mathfrak{R}$. This leads to the following theorem.
\begin{theorem}
\label{fixat}
Let $\rho$ be the representation of $G(\mathcal{O})$ that is the tensor product of $\rho_v$ as we constructed before Proposition \ref{623} for $v\in S$ and the trivial representation for $v$ outside $S$.

There is a bijection between the set $E_n^{irr}(\mathfrak{R})^{\Frob^{*}}$ and the set of inertial equivalence classes of cuspidal automorphic representations $\pi$ of $ G(\mathbb{A})$ such that \[\pi_{\rho}\neq \{0\}.\] 
\end{theorem}
\begin{proof}
The global Langlands correspondence is compatible with the local Langlands correspondence (\cite[Corollaire VII.5, p.197]{Lafforgue}) and preserves inertial equivalence relation (\cite[Théorème VI.9 (ii), p.158]{Lafforgue}). The result then follows from the global Langlands correspondence (\cite[Théorème VI.9 (i), p.158]{Lafforgue}), Proposition \ref{weilrep} and Proposition \ref{623}. 
\end{proof}

We will need the following result in later sections, which give an automorphic counterpart of the fact that 
\begin{equation}\label{prod1}
\prod_{x\in \overline{S}} \det \sigma_x=1,   \end{equation}
as a representation of $\hat{\mathbb{Z}}^{p'}(1)$. 

\begin{prop}\label{centralc}
The central character of $\rho$ is trivial restricting to $Z_{G}(\mathbb{F}_q)$. 
\end{prop}
\begin{proof}
The argument is essentially the same as in sections 3.1 and 3.2 of \cite{Deligne}.

Given $v\in S$, let $x_0\in \closure{S}$ be the point over $v$ that we have embeded $F_{1, (x_0)}$ in $\closure{F}$. 
As $\det \sigma_{x_0}$ is a character of  $\hat{\mathbb{Z}}^{p'}(1)$ that is fixed by $\Frob^{\ast d_v}$, there is a character $\chi_v: \mathbb{F}_{q_v}^\times\rightarrow \closure{\mathbb{Q}}_\ell $ such that $\det \sigma_{x_0}$ can be expressed as the composition: 
\[\det \sigma_{x_0}: \hat{\mathbb{Z}}^{p'}(1)\longrightarrow \mathbb{F}_{q_v}^\times \overset{\chi_v}{\longrightarrow} \closure{\mathbb{Q}}_{\ell}^\times. \]

The characters $\det \sigma_x$ for $x\in \closure{S}$ lying over $v$ satisfy the relation \[ (\det \sigma_{\Frob(x)})^q=\det \sigma_{x}.\] 
Therefore for each $v$, \[ \prod_{x\mid v}  \det \sigma_{x}=(\det \sigma_{x_0})^{\frac{q_v-1}{q-1}}.  \]
Note that $(\det \sigma_{x_0})^{\frac{q_v-1}{q-1}}$ can be factored as follows:
\begin{equation}\label{AB}
(\det \sigma_{x_0})^{\frac{q_v-1}{q-1}}: \hat{\mathbb{Z}}^{p'}(1)\longrightarrow \mathbb{F}_{q_v}^\times \overset{ N_{\mathbb{F}_{q_v}/\mathbb{F}_q } }{\longrightarrow} \mathbb{F}_{q}^\times \overset{\chi_v|_{\mathbb{F}_q^\times}}{ \longrightarrow  } \closure{\mathbb{Q}}_{\ell}^\times, \end{equation}
Since the composition of the first two maps in \eqref{AB} gives the quotient map from $\hat{\mathbb{Z}}^{p'}(1)$ to $\mathbb{F}_{q}^\times$, the condition \eqref{prod1} is equivalent to  
\[ \prod_v\chi_v|_{\mathbb{F}_q^\times}=1.  \]
Furthermore, using the construction of $\rho_v$ and applying the same arguments as before, we see that the central character of $\rho_v$ is the character $\chi_v$ of $Z_G(\mathbb{F}_{q_v})\cong \mathbb{F}_{q_v}^\times$. Therefore, the proposition follows. \end{proof}

\begin{prop}\label{Autcusp}
Suppose that the $\mathfrak{R}$ is generic:
\[E_n(\mathfrak{R})=E_n^{irr}(\mathfrak{R}). \] 
Then given any irreducible automorphic representation $\pi$ of $G(\AAA)$, \[\pi_{\rho}\neq \{0\}\implies \pi \text{ is cuspidal }.\]
For such a cuspidal automorphic representation $\pi$ and  any character $\chi$ that factors through $\deg\circ \det$, we have \[\pi\otimes\chi\cong \pi \implies \chi=1.\]

\end{prop}
\begin{proof}
Suppose that $\pi$ is not cuspidal. Then it is an irreducible subquotient of the parabolic induction of a cuspidal automorphic representation $\sigma$ of $M(\AAA)$ for a proper standard Levi subgroup $M$ of the proper parabolic subgroup $P$ of $G$. 

Suppose $\pi=\otimes'\pi_v$ is a restricted tensor product decomposition of $\pi$. 
If $\pi_{\rho} \neq 0$, then for any $v \in S$, we have \[ \pi_{v, \rho_v} \neq 0,\] and $\pi_v$ is unramified for $v$ outside $S$. Let $v \in S$ and $T_v$ be a maximal torus defined over $\kappa_v$ that we used in the Deligne-Lusztig induction to define $\rho_v$. 
After Moy-Prasad, the supercuspidal support of $\sigma_v$ can be determined in the following way.  Let $M_v$ be the minimal Levi subgroup of $G$ defined over $\kappa_v$ containing $T_v$, and $\tau_v$ be an extension of $\rho_v$ to $Z_{M_v}(F_v)M_v(\ooo_v)$, which is the normalizer of $M_v(\ooo_v)$ in $M_v(F_v)$ by Cartan decomposition. The representation \[ \phi_v = \cInd_{Z_{M_v}(F_v)M_v(\ooo_v)}^{M_v(F_v)}(\tau_v) \] is irreducible and supercuspidal (\cite[Proposition 6.6]{MP2}). It is shown in \cite[Theorem 6.11(2)]{MP2} that, for some extension $\tau_v$ of $\rho_v$, the $G(F_v)$-conjugacy class $[(M_v, \phi_v)]$ is the supercuspidal support of $\pi_v$.

For each place $v\in S$, the local component $\pi_v$ is also a subquotient of $\mathrm{Ind}_{P(F_v)}^{G(F_v) }\sigma_v$. Therefore there is an element $x_v\in G(F_v)$ such that $x_vM_v x_v^{-1} \subseteq M$ and $\sigma_v$ has supercuspidal support $[(x_vM_v x_v^{-1}, \phi_v^{x_v})]$. Since $P$ is defined over $\mathbb{F}_q$ and $M_v$ is defined over $\kappa_v$, we can take $x_v\in G(\kappa_v)$. 
Up to a conjugation of $G(\kappa_v)$, we may suppose that $M_v\subseteq M$ (here we use base change to view $M$ as a group over $\kappa_v$) and $\sigma_v$ has supercuspidal support $[(M_v, \phi_v)]$. Let $\rho^M_v$ be the irreducible representation of $M(\kappa_v)$ whose irreducible character is equal to the Deligne-Lusztig induced character $R_{T_v}^{M}(\theta_v)$ up to a sign, and $\rho^M$ the representation of $M(\ooo)$ that is the tensor product over all $v\in S$ of $\rho^M_v$, then we have \[\sigma_{\rho^M}\neq\{0\}. \]
 Applying the global Langlands correspondence and Corollary \ref{623} to $M$, which is a product of general linear groups, we obtain a direct sum of $\ell$-adic local systems whose local monodromies are prescribed by $\mathfrak{R}$. This is a counter-example to the genericity condition that 
 \[E_n^{irr}(\mathfrak{R})=E_n(\mathfrak{R}). \]

 The second part is similar. After a twist, we can suppose that the central character of $\pi$ has finite order, corresponding to an $\ell$-adic local system over $X-S$ by the Langlands correspondence.   
 If \[ \pi \otimes \chi\cong \pi,\]
 then by \cite[Proposition 2.1.3]{Yu1}, the $\ell$-adic local system over $X-S$ corresponding to $\pi$ is not irreducible restricting to $\closure{X}-\closure{S}$. This contradicts our hypothesis that 
$\mathfrak{R}$ is generic. 
\end{proof}

\begin{theorem}\label{mul2}
We define 
$$e_{\rho}= \bigotimes_{v\in S} e_{\rho_v}\otimes \bigotimes_{v\notin S}\mathbbm{1}_{G(\ooo_v)}\in C_c^\infty(G(\AAA)),$$
to be the tensor product of the the characteristic function $\mathbbm{1}_{G(\mathcal{O}_v)}$ of ${G(\ooo_v)}$ for $v\notin S$ with $e_{\rho_v}$ defined below for $v\in S$:
$$e_{\rho_v}(x):=\begin{cases}
\Theta_{\rho_v}(\closure{x}^{-1}), \quad x\in G(\ooo_v);\\
0, \quad x\notin G(\ooo_v); 
\end{cases}$$
where $\Theta_{\rho_v}$ is the trace function of the representation $\rho_v$ of $G(\kappa_v)$ and $\closure{x}$ denotes the image of $x$ via the reduction map $G(\ooo_v)\rightarrow G(\kappa_v)$. For any irreducible automorphic representation $\pi$, if $\pi_\rho\neq 0$, we have 
\[ \mathrm{Tr}(e_{\rho}|\pi)=1; \]
if $\pi_\rho= 0$, we have 
\[ \mathrm{Tr}(e_{\rho}|\pi)=0. \]
\end{theorem}
\begin{proof}
Let $\pi$ be an irreducible automorphic representation of $G(\mathbb{A})$. Then $\pi$ admits a restricted tensor product decomposition of the form $\pi=\bigotimes' \pi_v$, where the tensor product is taken over all places $v$ of $F$, and $\pi_v$ is an irreducible smooth representation of $G(F_v)$. 
If for every $v\notin S$, $\pi_v$ is unramified, i.e., if $\pi_v^{G(\mathcal{O}_v)}\neq 0$, then 
\[\Tr(e_\rho|\pi)=\prod_{v\in S} \Tr(e_{\rho_v}|\pi_v). \]
Otherwise \[\Tr(e_\rho|\pi)=0.\]
Note that \[ \Tr(e_{\rho_v}|\pi_v)=\dim_{\mathbb{C}}\Hom_{G(\mathcal{O}_v)}(\rho_v, \pi_v),\] it suffices to show that for every $v\in S$: 
\[\dim_{\mathbb{C}}\Hom_{G(\mathcal{O}_v)}(\rho_v, \pi_v)\leq 1. \]

Suppose that the cuspidal support of $\pi_v$ is the conjugacy class of $(P, \sigma)$, where $P$ is a standard parabolic subgroup of $G$, and $\sigma$ is a cuspidal representation of the standard Levi subgroup $M$ of $P$. It suffices to prove that  
\[\dim_{\mathbb{C}}\Hom_{G(\mathcal{O}_v)}(\rho_v, \Ind_{P}^G\sigma)\leq 1.\]
Let $G(\mathcal{O}_v)_{+}$ be the pro-p radical of $G(\mathcal{O}_v)$. 
Since $\rho_v$ is induced from a representation of $G(\kappa_v)$, we deduce that 
\[\Hom_{G(\mathcal{O}_v)}(\rho_v, \Ind_{P}^G\sigma) \cong \Hom_{G(\kappa_v)}(\rho_v, (\Ind_{P}^G\sigma)^{G(\mathcal{O}_v)_{+}}).  \]
Here $(\Ind_{P}^G\sigma)^{G(\mathcal{O}_v)_{+}}$ is the ${G(\mathcal{O}_v)_{+}}$-fixed subspace of $\Ind_{P}^G\sigma$. 
The above Hom sets are further isomorphic to 
\[\Hom_{G(\kappa_v)}(\rho_v, \Ind_{P(\kappa_v)}^{G(\kappa_v)}(\sigma^{M(\mathcal{O}_v)_{+}})     ).   \]
By \cite[Proposition 6.8]{MP2} (or by the discussions in Section \ref{Expli}), we know that there is an irreducible cuspidal representation $\lambda$ of $M(\kappa_v)$ so that 
$\sigma$ is compactly induced from a representation $\Lambda$ of $Z_{M}(F_v)M(\mathcal{O}_v)$ that inflates and extends $\lambda$. The $M(\mathcal{O}_v)_{+}$-invariant subspace of $\sigma$ is isomorphic to $\Lambda$ as a representation of $Z_{M}(F_v)M(\mathcal{O}_v)$ (see \cite[Corollary 5.3]{Vig}). In particular, it is irreducible. We have 
\[ \Hom_{G(\mathcal{O}_v)}(\rho_v, \Ind_{P}^G\sigma) \cong \Hom_{G(\kappa_v)}(\rho_v, \Ind_{P(\kappa_v)}^{G(\kappa_v)}\lambda). \]
Since $\lambda$ is a cuspidal representation of a finite Levi subgroup of $G(\kappa_v)$, it is a Deligne-Lusztig induced representation  (by Green's classification for representations of finite general linear groups \cite{Green}). Therefore $\Ind_{P(\kappa_v)}^{G(\kappa_v)}\lambda$ is also a Deligne-Lusztig induced representation (see the remark below \cite[Definition 11.1]{DiMi} and Paragraph 11.5 of \textit{op. cit.}). 
We deduce by \cite[Theorem 6.8]{DL} that
\[\dim_{\mathbb{C}} \Hom_{G(\kappa_v)}(\rho_v, \Ind_{P(\kappa_v)}^{G(\kappa_v)}\lambda) \leq 1,\]
as $\rho_v$ is induced from a character in general position. 
\end{proof}

\section{The Arthur-Selberg trace formula}

In this section, we introduce the Arthur-Selberg trace formula established over a function field by L. Lafforgue \cite{Laf}. It is a linear functional over $C_c^\infty(G(\AAA))$. We modify the construction by adding an additional parameter $\xi$. This construction is inspired by the semistability of the parabolic structures of Higgs bundles. We demonstrate in Theorem \ref{Einfinitesimal} that the cardinality of the set $E_n^{irr}(\mathfrak{R})^{\Frob^\ast}$ can be calculated using these trace formulas with the test function constructed in Theorem \ref{mul2}. To deal with them, we also introduce an infinitesimal trace formula and relate the original trace formula to this infinitesimal analog.

\subsection{Notation for the trace formula}
Let $B$ be the Borel subgroup of $G$ consisting of upper triangular matrices, and $T$ be the maximal torus of diagonal matrices. Let $W=W(G, T)$ be the Weyl group of $G$ related to $T$. We identify $W$ with the set of permutation matrices. Let $\mathfrak{g}$, $\mathfrak{b}$, $\mathfrak{n}$ and $\mathfrak{t}$ denote respectively the Lie algebras of $G$, $B$, the unipotent radical of $B$ and $T$.

Let $R$ be an algebraic subgroup of $G$ defined over $F$, we  define
\[ X^{*}(R)=\Hom_F(R,\mathbb{G}_{m/F}),\] 
and
\[ \ago_R^{*}= X^{*}(R)\otimes\mathbb{Q},\quad \ago_R= \Hom_{\mathbb{Q}}(\ago_R^{*}, \mathbb{Q}).\] 
%Let $A_R$ be the maximal split central torus defined over $F$ of $R$. We should think of $X^*(R)\subseteq X^*(A_R)$ as lattices in $\ago_R^{*}$ and think of $\ago_R^{*}$ as the space of linear functions on $\ago_R$. 

Suppose that $P, Q$ are standard parabolic subgroups of $G$. 
 We will use the following notation. 
\begin{enumerate}
\item[$\bullet$] $\mathcal{P}_0$: the set of standard parabolic subgroups of $G$;%, i.e. parabolic subgroups defined over $F$ containing $B$; 
\item[$\bullet$] $\langle \cdot, \cdot \rangle:$ the canonical pairing between $\ago_B^{*}$ and $\ago_B$;
\item[$\bullet$] $\ago_P=\ago_P^Q\oplus \ago_Q$ ($P\subseteq Q$), decomposition induced by $X^{*}(Q)\rightarrow X^{*}(P)$ and $X^{*}(A_P)\rightarrow X^{*}(A_Q) $, where in our case that $G=GL_n$, $A_P$ (resp. $A_Q$) is the center of $P$ (resp. $Q$);
\item[$\bullet$] $\Phi(P,T)$: roots of the torus $T$ acting on the Lie algebra of $\ppp$ of $P$;
\item[$\bullet$] $\Delta_B$ simple roots in $ \Phi(G, T)$; 
\item[$\bullet$] $\Delta_B^P$ simple roots in $ \Phi(M_P, T)$, where $M_P$ is the standard Levi subgroup of $P$; 
\item[$\bullet$] $\Delta_B^{\vee}$ the set of corresponding coroots; $\hat{\Delta}_B$ the set of corresponding fundamental weights (the base that is dual to $\Delta_B^\vee$); 
\item[$\bullet$] $\hat{\Delta}_P=\{\varpi \in \hat{\Delta}_B  \mid \varpi|_{\ago_B^P}=0  \}$;
\item[$\bullet$] $\Delta_P^Q=\{\alpha|_{A_P} \mid \alpha \in \Delta_B^Q-\Delta_B^{P}\}$, viewed as a set of linear functions on $\ago_B$ via the projection $\ago_B\rightarrow \ago_P^Q$;
\item[$\bullet$] $\htau_P$ the characteristic function of $H\in \ago_B$ such that $\langle\varpi,H\rangle>0$ for all $\varpi\in \hat{\Delta}_P$; 
\item[$\bullet$] $\tau_P^{Q}$ the characteristic function of $H\in \ago_B$ such that $\langle\alpha,H\rangle>0$ for all $\alpha\in \Delta_P^Q$; 
\end{enumerate}
Note that since $T\cong \mathbb{G}_m^n$, we have \[\ago_B\cong \mathbb{Q}^n. \]
If one prefers, there are explicit presentations using a coordinate system for the above notation. The reader can consult \cite[2.5]{Chau} or \cite[3.1]{Yu1}.

For each semi-standard parabolic subgroup $P$ of $G$ defined over $F$, we define the Harish-Chandra's map $H_P: P(\AAA)\rightarrow \ago_{P} $ by requiring   
\begin{equation}\langle\alpha,  H_P(p)\rangle =\deg \alpha(p),  \quad \forall \alpha\in X^{*}(P)_F.  \end{equation}
Using Iwasawa decomposition, we may extend the definition of $H_P$ to $G(\AAA)$ by asking it to be $G(\ooo)$-right invariant.

Let $x\in G(\AAA)$, for each place $v$. We define a Weyl element $s_{x,v}\in W$ as follows. First, we have the Iwasawa decomposition $x=pk$ with $p\in B(\AAA)$ and $k\in G(\ooo)$. Then we define $s_{x,v}$ as the image of $k$ via in the following double quotient identified with $W$: 
\[\mathcal{I}_v \backslash G(\ooo_v )/ \mathcal{I}_v \cong B(\kappa_v)\backslash G(\kappa_v)/B(\kappa_v)\cong W.\]

The Haar measures are normalized in the following way.
For any place $v$ of $F$ and 
for a Lie algebra $\ggg$ defined over $\mathbb{F}_q$, the volumes of the sets $\ggg(\ooo_v)$ and $\ggg(\ooo)$ are normalized to be $1$; the volume of $G(\ooo_v)$ and $G(\ooo)$ are normalized to be $1$; for every semi-standard parabolic subgroup $P$ of $G$ defined over $\mathbb{F}_q$, the measure on $N_P(\AAA)$ is normalized so that \[ \vol(N_P(F)\backslash N_P(\AAA))=1.\]

\subsection{Geometric side of the Arthur-Selberg trace formula}\label{relationst}
In this section, we will recall the geometric side of the trace formula, while omitting the spectral side. Due to the generic condition we have imposed, the spectral side of the trace formula is considerably simplified and will be discussed in the last section.

Let $a\in \mathbb{A}^{\times}$ be a degree 1 idèle, viewed as a scalar matrix in $G(\mathbb{A})$. 
Let $f\in C_{c}^{\infty}(G(\AAA))$ be a complex smooth function with compact support. Let $P$ be a standard parabolic subgroup of $G$. The function $f$ acts via the right regular representation on $L^2(M_P(F)N_P(\AAA)\backslash G(\AAA)/a^{\mathbb{Z}})$, where $M_P$ is the standard Levi subgroup of $P$ and $N_P$ is the nilpotent radical of $P$. 
 It is an integral operator with the kernel function over $G(\AAA)\times G(\AAA)$ given by:
\[  k_P(x,y)= \sum_{\gamma\in M_P(F)a^{\mathbb{Z}}}\int_{N_P(\AAA)} f(y^{-1}\gamma n x)\d n. \]

Let \[\xi=(\xi_v)_{v\in S}\in (\ago_{B})^{S}\cong (\mathbb{Q}^n)^S.\] 
The Arthur-Selberg trace formula calculates the integral over $x\in G(F)\backslash G(\AAA)^{e}$ of the following expression: 
\begin{equation}\label{truncatedk} 
k_f^{\xi}(x):=\sum_{P\in \mathcal{P}_0} (-1)^{\dim \ago_P^{G}} \sum_{\delta\in P(F)\backslash G(F)}\hat{\tau}_P(H_B(\delta x) + \sum_{v\in S}s_{\delta x, v}\xi_v) {k}_{P}(\delta x, \delta x).  \end{equation}
The summand for $P=G$ is $k_G(x, x)$, and the other terms ensure that the sum is integrable (it is proved in \cite{Laf} when $\xi=0$ and in \cite{Yu2} with a non-zero $\xi$).  We denote this integral by \[J^{e, \xi}(f). \]
When $\xi=0$, we will drop it from the notation.

For our purpose, it is easier to use another expression for the truncated trace.  We define the function $j_f^\xi$ over $G(\AAA)$ by:
\[j_f^\xi(x):=\sum_{P\in \mathcal{P}_0} (-1)^{\dim \ago_P^{G}} \sum_{\delta\in P(F)\backslash G(F)}\hat{\tau}_P(H_B(\delta x) + \sum_{v\in S}s_{\delta x, v}\xi_v) \widetilde{k}_{P}(\delta x, \delta x),  \] 
where \[  \widetilde{k}_P(x,y)= \sum_{\gamma\in P(F) a^{\mathbb{Z}}}f(y^{-1}\gamma x). \] 
We proved in \cite[Proposition 9.1]{Yu2} that $j_f^{\xi}$ is integrable, and its integration over $G(F)\backslash G(\AAA)^{e}$ coincides with $k_f^\xi$:
\begin{equation}J^{e, \xi}(f)=\int_{G(F)\backslash G(\AAA)^e}j_f^{\xi}(x)\d x. \end{equation}

\subsection{Expressing $|E_n^{irr}(\mathfrak{R})^{\Frob^\ast}|$ using a trace formula}
The goal of this subsection is to express the cardinality $|E_n^{irr}(\mathfrak{R})^{\Frob^\ast}|$ using a trace formula.

\begin{theorem}\label{Einfinitesimal}

For any $(e, \xi)\in \mathbb{Z}\times(\ago_B)^S$, we have:
\[J^{e,\xi}(e_\rho)= |E_n^{irr}(\mathfrak{R})^{\Frob^{\ast}}|.\]
In particular, the left-hand side is independent of $(e, \xi)$. 
\end{theorem}

\begin{proof}
We have chosen a lattice $a^\mathbb{Z}$ in $Z_G(F)\backslash Z_G(\AAA)$ generated by a scalar id\`ele $a$ of degree $1$. Let $\zeta$ be a primitive $n$-th root of unity, and $\eta$ be the character of $G(\AAA)$ defined by
\[ \eta(x)= \zeta^{\deg \det x}.\] We define $J^{\xi}_{\eta}(e_\rho)$ as follows: 
$$J^{ \xi}_{\eta}(e_\rho)= \int_{G(F)\backslash G(\AAA)/a^\mathbb{Z}} \eta(x) k^{\xi}_{e_\rho}(x)\d x.   $$
Using this definition, for any integer $e\in \mathbb{Z}$, we deduce that \[J^{e, \xi}(e_\rho) = \frac{1}{n}\sum_{k=1}^{n}\zeta^{-ek}J^{\xi}_{\eta^{k}}(e_\rho).  \]

Recall that the function $k^{\xi}_{e_\rho}(x)$ is a truncated linear combination of the diagonal of the kernel functions $k_{P}(x,y)$ of $e_\rho$ acting via the regular representation on $L^2(N_P(\AAA)M_P(F)\backslash G(\AAA)/ a^\mathbb{Z})$. The operator $R(({\dim \rho}) e_\rho)$ is a projector onto the $\rho$-isotypic subspaces, as can be seen from the identity \[({\dim \rho}) e_\rho\ast ({\dim \rho}) e_\rho=({\dim \rho}) e_\rho,\]
where $\ast$ is the convolution product.

When $P\neq G$, no irreducible subquotient in $L^2(N_P(\AAA)M_P(F)\backslash G(\AAA)/ a^\mathbb{Z})$ is cuspidal. Hence, by Proposition \ref{Autcusp} and Proposition \ref{mul2}, $R(e_\rho)$ is the zero operator. For $P=G$, the image of $R(e_\rho)$ is contained in the subspace of cusp forms $L^2_{cusp}$ of $L^2(G(F)\backslash G(\AAA)/a^\mathbb{Z})$. Therefore, 
\[ J^{\xi}_{\eta^k}(e_\rho)=\Tr(\eta^k \circ R(e_\rho)\mid L^2_{cusp}).\]

The multiplicity one theorem  of Jacquet, Langlands, Piatetski-Shapiro, and Shalika, implies that
\[ L^2_{cusp}=\hat{\bigoplus}_{\pi} \pi , \]
where the Hilbert direct sum is taken over the set of isomorphism classes of cuspidal automorphic representations $\pi$, such that the central character $\chi_\pi$ is trivial on $a^\mathbb{Z}$. For each $\pi$, let $l_\pi$ be the smallest strictly positive integer such that \[ \pi\cong \pi\otimes{\eta^{l_\pi}}. \]
If $\pi$ appears in $L^2_{cusp}$, then there exist $l_\pi$ distinct isomorphism classes of cuspidal automorphic representations in $L^2_{cusp}$ that are inertially equivalent to $\pi$. These classes correspond to twists of $\pi$ by a power of the character $\eta$.
We deduce that
\begin{equation}\label{001}J^{\xi}_{\eta^k}(e_\rho)= \sum_{\pi} \Tr( \eta^k \circ R(e_\rho) | \pi\oplus \pi_{\eta}\oplus\ldots\oplus \pi_{\eta^{l_\pi-1}}),  \end{equation}
where the sum is taken over the set of inertial equivalence classes of cuspidal automorphic representations $\pi$ whose central character $\chi_\pi$ is trivial on $a^\mathbb{Z}$. Here $\pi_\eta:=\pi\otimes \eta $.
Since $\eta^k\circ R(e_\rho)$ maps an element in $\pi$ to $\pi_{\eta^k}$, it follows that if $\pi_{\eta^k}\not\cong \pi$, then we have \[\Tr( \eta^k \circ R(e_\rho) | \pi\oplus \pi_{\eta}\oplus\ldots\oplus \pi_{\eta^{l_\pi-1}})=  0. \] 
If $\pi_{\eta^k}\cong \pi$, then $l_\pi\mid k$ and
\[\Tr( \eta^k \circ R(e_\rho) | \pi\oplus \pi_{\eta}\oplus\ldots\oplus \pi_{\eta^{l_\pi-1}})=l_\pi \Tr(\eta^k \circ R(e_\rho)| \pi).  \]
Therefore we have
\begin{equation}\label{001}J^{\xi}_{\eta^k}(e_\rho)= \sum_{\pi: l_\pi\mid k} l_\pi \Tr(\eta^k \circ R(e_\rho)\mid \pi). \end{equation}
Here the sum is taken over the set of inertial equivalence classes of cuspidal automorphic representations $\pi$ such that $\pi\otimes\eta^k\cong \pi$.

If $\pi$ does not contain $\rho$, then $R(e_\rho)$ is the zero operator on $\pi$ and we have
\[ \Tr(\eta^k \circ R(e_\rho)\vert \pi) =0, \quad\forall k\geq 1.\]
If $\pi$ contains $\rho$, then Proposition \ref{Autcusp} says $l_\pi=n$, and Proposition \ref{mul2} implies
\[\Tr(R(e_\rho)\vert \pi)=1. \]
As a consequence, we have 
\[ J_{\eta^k}^{\xi}(e_\rho)=\begin{cases} 0, \quad n\nmid k; \\  n\sum_{\pi: \pi_{\rho}\neq 0}1, \quad n\mid k. \end{cases}  \]
By Theorem \ref{fixat}, it follows that for any $(e,\xi)\in \mathbb{Z}\times (\ago_B)^S$, 
\begin{equation}
J^{e,\xi}(e_\rho)= |E_n^{irr}(\mathfrak{R})^{\Frob^{\ast}}|.
\end{equation}
This completes the proof.
\end{proof}

\subsection{A trace formula for Lie algebra}  \label{Fourier}
The key to our method of treating the Arthur-Selberg trace formula is passing to Lie algebra. An infinitesimal analogue of the Arthur-Selberg trace formula has been established in \cite{Yu2} over a function field following the strategy of the number field case established in \cite{ChauLie}. 

Let $f\in \mathcal{C}_c^{\infty}(\ggg(\AAA))$. The Lie algebra analogue of the truncated kernel function is defined  by:
 \[ j^{\ggg}_{f}(x):= \sum_{P\in \mathcal{P}_0}(-1)^{\dim \ago_P^{G}}\sum_{\delta\in P(F)\backslash G(F)}\hat{\tau}_P(H_B(\delta x) + \sum_{v\in S}s_{\delta x, v}\xi_v) \widetilde{k}^{\ggg}_{P}(\delta x),   \] 
where $\mathfrak{p}$ is the Lie algebra of $P$ and $\widetilde{k}^{\ggg}_P(x)$ is defined by 
\[\widetilde{k}^{\ggg}_P(x)= \sum_{\gamma \in \ppp(F)} f(\Ad(x^{-1})(\gamma)). \]

Let $\omega_X$ be the canonical line bundle of $X$ and $K_X$ a divisor associated to it, i.e. $\mathcal{O}_X(K_X)\cong \omega_X$. Suppose that $K_X=\sum_{v} d_v v $. 
We have (\cite[Section 4]{Tate})
$$\frac{\mathbb{A}}{F+ \prod_v\wp_v^{-d_v}\mathcal{O}_v}\cong H^{1}(X, \omega_X)\cong H^{0}(X,\mathcal{O}_{X})^{*}\cong\mathbb{F}_{q},$$ by Serre duality and the fact that $X$ is geometrically connected. 
We fix a non-trivial character $\psi_0$ of $\mathbb{F}_q$. 
Via the above isomorphisms, $\psi_0$ defines a character $\psi$ of $\AAA/F$. %We use this $\psi$ in the definition of Fourier transform. 
For $f\in C_c^\infty(\ggg(\AAA))$, the Fourier transform of $f$ is defined by the function 
\[\hat{f}(x):= \int_{\ggg(\AAA)}f(y)\psi(\Tr(xy))\d y, \]
where $\Tr(xy)$ is the trace of the product of $x,y$ as two matrices. 
We can decompose $$\psi=\otimes_{v}\psi_{v}, $$
with $\psi_v: F_v\rightarrow \mathbb{C}^{\times}$. It is clear that $\psi_v$ is trivial on $\wp_{v}^{-d_v}$ but not on $\wp_{v}^{-d_v-1}$. 
We define a local Fourier transformation for functions in $\mathcal{C}_{c}^{\infty}(\ggg(F_v))$ using the character $\psi_v$. 
 It is clear by definition that for any $f=\otimes_v f_v$, we have \[\hat{f}=\otimes_v \hat{f_v }. \]

The infinitesimal trace formula is the following identity that applies for any $f\in C_c^{\infty}(\ggg(\AAA))$ (Theorem 5.7 of \cite{Yu2}):
\begin{equation}\label{ITF}
J^{\ggg, e, \xi}(f)=q^{-(g-1)n^2}J^{\ggg, e, \xi}(\hat{f}).\end{equation}

\subsection{A coarse expansion}\label{CoEx}
We need to explore a coarse expansion and a vanishing result. In the case that the parameter $\xi\in (\ago_B)^S$ is set to be zero, the result is due to Chaudouard (for the coarse expansion and vanishing theorem \cite[Section 6]{Chau}). %We refer to \cite[Section 5, 9]{Yu2} when $\xi$ is non-trivial. 

The coarse expansion is the decomposition of $J^{e,\xi}(f)$ based on the characteristic polynomials of elements in $G(F)$: 
 \[ J^{e,\xi}(f)=\sum_{\chi } J^{e,\xi}_{\chi}(f), \]
where $J^{e, \xi}_{\chi}(f)$ is the integral over $G(F)\backslash G(\AAA)^e$ of the function 
 \[ j_{f,\chi}^{\xi}(x):=\sum_{P\in \mathcal{P}_0} (-1)^{\dim \ago_P^{G}} \sum_{\delta\in P(F)\backslash G(F)}\hat{\tau}_P(H_B(\delta x)+\sum_{v\in S}s_{\delta x, v}\xi_v) \widetilde{k}_{P, \chi}(\delta x, \delta x). \]
In the above expression,  \[\widetilde{k}_{P,\chi}(x,x)= \sum_{\gamma\in P(F)_\chi a^\mathbb{Z}} f (x^{-1}\gamma x), \]
where $P(F)_\chi $ is the set of elements in $P(F)$ whose characteristic polynomial is $\chi$. 

Similarly, we have a coarse expansion of $J^{\ggg, e, \xi}$: 
 \[ J^{\ggg, e,\xi}(f)=\sum_{\chi} J^{\ggg, e,\xi}_{\chi}(f), \]
where the sum is taken over the set of characteristic polynomials of elements in $\ggg(F)$. 
%Comparing with the definition of $\widetilde{k}_{P}(x,x)$. %, note that we have ignored the sum over $a^\mathbb{Z}$ which is because the support of $e_\rho$ is contained in $G(\mathcal{O})$ and only $1\in a^\mathbb{Z}$ contributes.  

The image of Harish-Chandra's map $H_P$ from $G(\AAA)$ to $\ago_P$ is the lattice $X_\ast(P)$, defined to be the dual of $X^{\ast}(P)$. For $P\subseteq Q$, the natural morphism $X_\ast(P)\rightarrow X_{\ast}(Q)$ makes the diagram commutes:
\[\begin{tikzcd}[column sep=-.7em]
  & G(\AAA) \arrow[dl, "H_P"'] \arrow[dr, "H_Q"] & \\
  X_{\ast}(P) \arrow[rr, ""'] & & X_{\ast}(Q)
\end{tikzcd}\]
For $e\in \mathbb{Z}\cong X_{\ast}(G)$ ($1$ is the element that sends $\det\in X^{\ast}(G)$ to $1$), we denote by $X_\ast(P)^e$ the preimage of $e$. It is the image $G(\AAA)^e$ under the map $H_P$. 
 \begin{definition}\label{adgen}
 We say that a pair \[ (e, \xi)\in X_{\ast}(G)\times (\ago_{B})^S\cong \mathbb{Z}\times (\mathbb{Q}^n)^S \]
is generic if the projection of $\sum_{v} s_v \xi_v$ in $\ago_P$ does not belong to $X_{\ast}(P)^{e}+\ago_G$ for any proper semi-standard parabolic subgroup $P$ of $G$ and any $s_v\in W$ ($v\in S$).  
\end{definition}
\begin{remark}\label{REMgen} \normalfont
It is clear from the definition that if $S\neq \emptyset$, for any $e\in \mathbb{Z}$, most $\xi\in (\ago_{B})^S$ are generic.

If $S=\emptyset$, $e\in \mathbb{Z}$ is generic if and only if $(e, n)=1$. In fact, if $(e,n)\neq 1$, we can eaily see that $0\in X_{\ast}(P)^{e}+\ago_G$, for the standard parabolic subgroup $P$ whose Levi factor is $GL_{n/(n,e)}^{(n,e)}$. If $(n,e)=1$, the arguments in \cite[Lemme 6.2.6]{Chau} show that $e$ is generic. 
\end{remark}

 The following result generalizes \cite[Théorème 6.2.1]{Chau}. For a place $v$, let $\mathcal{I}_v$ be the standard Iwahori subgroup of $G(\mathcal{O}_v)$ and  $\mathfrak{I}_v$ be the standard Iwahori subalgebra of $\ggg(\mathcal{O}_v)$. 
 \begin{theorem}\label{vanishing}
Suppose that $(e, \xi)\in \mathbb{Z}\times (\ago_B)^S$ is generic (including the case that $S=\emptyset$ where genericity means that $(e,n)=1$ as observed in Remark \ref{REMgen}) and $S\subseteq X(\mathbb{F}_q)$. 

Suppose that $f\in C_c^\infty(G(\AAA))$ is supported in $G(\mathcal{O}^S)\times\prod_{v\in S}\mathcal{I}_v$. Let $\chi\in F[t]$ be a degree $n$ monic polynomial with a non-zero constant. If $\chi$ is not of the form $(t-\alpha)^n$ for some  $\alpha\in \mathbb{F}_q^\times$, then
\[ J_{\chi}^{e,\xi}(f)=0. \]

Suppose that $\varphi\in C_c^\infty(G(\AAA))$ is supported in $\ggg(\mathcal{O}^S)\times\prod_{v\in S}\mathfrak{I}_v$. Let $\chi\in F[t]$ be a degree $n$ monic polynomial. If $\chi$ is not of the form $(t-\alpha)^n$ for some $\alpha\in  \mathbb{F}_q$, then
\[ J_{\chi}^{\ggg, e,\xi}(\varphi)=0. \]
 \end{theorem}
\begin{proof}
Our proof adopts Chaudouard's strategy and extends it to include the treatment of an additional parameter $\xi$. We only present the proof for the group version, as the same method applies to the Lie algebra case.

First, we note that in the definition of $\widetilde{k}_{P,\chi}(x,x)$, the sum over $a^\mathbb{Z}$ can be ignored because the support of $f$ is contained in the degree $0$ part.
Let $\chi$ be a monic polynomial of degree $n$ whose constant term is non-zero.   
Let $\sigma\in G(F)$ be a semi-simple element with characteristic polynomial $\chi$.
We observe that if $J_{\chi}^{e,\xi}(f)\neq 0$, then the coefficients of the characteristic polynomial of $\sigma$ belong to $\mathcal{O}\cap F=\mathbb{F}_q$. In particular, it splits after a separable extension of $F$.  

We have the following lemma, a variant of Lemma 3.1 of \cite{Ageom}, and it can be proved similarly. The fact that $\chi$ splits over a separable extension implies that the Jordan-Chevalley decomposition holds for elements in $G(F)$ whose characteristic polynomial is $\chi$.  
\begin{lemm}
Let $\chi$ be a monic polynomial of degree $n$ whose constant term is non-zero which splits after a separable extension of $F$.     
Let $\sigma\in G(F)$ be a semi-simple element with characteristic polynomial $\chi$. Let $M_1$ be the minimal Levi subgroup that contains $\sigma$. We may choose $\sigma$ so that $M_1$ is standard. 
Then 
 \(j_{f,\chi}^{\xi}(x) \) equals the sum over standard parabolic subgroup $R$ of $G_{\sigma}$ and elements $\eta\in R(F)\backslash G(F)$ of the product of 
 \begin{equation} \sum_{u\in \mathcal{U}_{M_R}(F)}\sum_{n\in N_R(F)} f(x^{-1}\eta^{-1}\sigma un\eta x  ) \end{equation}
 with 
 \begin{equation}\label{sigmaP}\sum_{P\in\mathcal{F}_R(M_1)}(-1)^{\dim \ago_P^G}\htau_P(H_P(\eta x) + \sum_{v\in S}s^{-1}s_{w_s\eta x, v}\xi_{v} ),  \end{equation}
where $\mathcal{F}_R(M_1)$ is the set of parabolic subgroups $P$ of $G$ defined over $F$ containing $M_1$ and such that $P_\sigma=R$, $\mathcal{U}_{M_R}(F)$ is the set of unipotent element in $M_R(F)$, 
and $s\in W$ with representative $w_s\in G(\mathbb{F}_q)$ is an element such that $w_sPw_s^{-1}$ is standard.  
\end{lemm}
Let $P_1$ be the standard parabolic subgroup of $G$ whose Levi subgroup is $M_1$. Let $N_1$ be the unipotent radical of $P_1$. We suppose that $\sigma\in G(\mathbb{F}_q)$. 

If $J_{\chi}^{e, \xi}(f)\neq 0$, then for some standard parabolic subgroup $R$ of $G_\sigma$, there are $x\in G(\AAA)$, $\eta\in R(F)\backslash G(F)$, $n\in N_R(F)$ and $u\in \mathcal{U}_{M_R}(F)$ such that \begin{equation}\label{xeta} f(x^{-1}\eta^{-1}\sigma u n \eta x )\neq 0. \end{equation}
We can choose a representative $\eta\in G(F)$ so that $u\in N_B(F)\cap \mathcal{U}_{M_R}(F)$. We have $u\in B(F)\cap M_R(F)\subseteq P_1(F)\cap M_R(F)$.  Since $M_{1, \sigma}$ is the Levi subgroup of $P_1\cap M_R$ that is maximal anisotropic,  it does not contain any rational unipotent element. Thefore, we can conlude that $u\in N_1(F)\cap M_R(F)$.  
Let $\eta x=n_1m_1k$ with $n_1\in N_1(\AAA)$, $m_1\in M_1(\AAA)$ and $k\in G(\mathcal{O})$ be the Iwasawa decomposition of $\eta x$, then we have 
\[ m_1^{-1}\sigma m_1 \in M_1(\mathcal{O}). \]
By Kottwitz \cite[Proposition 7.1]{Ko3}, we deduce that $m_1^{-1}\sigma m_1 $ and $\sigma$ are conjugated by an element in $M_1(\mathcal{O})$. Therefore,
\[m_1\in M_{1, \sigma}(\AAA)M_1(\mathcal{O}). \]
By modifying $m_1$ and $k$, we can suppose that $m_1\in M_{1,\sigma}(\AAA)$. 
We have \begin{equation}\label{nsig} n^{-1}_1\sigma (m_1^{-1}  unm_1  ) n_1 \in  \sigma N_1(\mathcal{O}).   \end{equation}
The proposition \cite[Proposition 7.2]{Yu2} for the ring $A=\mathbb{A}$ says that the pair $([n_1], m_1^{-1}  unm_1)$ is unique for the image of $n_1$ in $N_{1,\sigma}(\AAA)\backslash N_1(\AAA)$. The same Proposition for the ring $A=\mathcal{O}$ implies that we can choose a representative of $[n_1]$ in $N_1(\mathcal{O})$. Therefore
\[n_1\in N_{1,\sigma}(\AAA)N_1(\mathcal{O}). \]
We conclude that \begin{equation}\eta x\in R(F)M_{1,\sigma}(\AAA)N_{1,\sigma}(\AAA)G(\mathcal{O}). \end{equation}

 Now we look at the component at a place $v\in S$. Let $\eta x$, $u$, and $n$ be as above. 
 We want to prove that \begin{equation}W^{Q} s_{w_s \eta x, v}=W^{Q} s s_{\eta x, v}, \end{equation}
 where $Q=w_{s}Pw_s^{-1}$ is standard. 

Suppose that $\eta x_v=rk$ with $r\in R(F_v)$ and $k\in G(\mathcal{O}_v)$. The condition  on the support of $f$ implies that, for $J_{\chi}^{e,\xi}(f)$ to be non-zero, we must have
\[k^{-1}\sigma k\in \mathcal{I}_v, \]
because the semisimple part of $(\eta x)^{-1}\sigma un \eta x\in \mathcal{I}_v$ lies in $\mathcal{I}_v$. Therefore, 
\[ \dot{k}^{-1}\sigma \dot{k}\in B(\kappa_v),  \]
where $\dot{k}$ is the image of $k$ in $G(\kappa_v)$
\begin{lemm}
If $\dot{k}\in G(\kappa_v)$ be an element such that 
\[ \dot{k}^{-1}\sigma \dot{k}\in B(\kappa_v).  \]
Then $\dot{k}\in P_{1,\sigma}(\kappa_v)WB(\kappa_v)$. 
\end{lemm}
\begin{proof}
The condition implies that $\sigma$ is contained in a split maximal torus over $\kappa_v$. Hence, there exists  $m\in M_1(\kappa_v)$ such that 
\[m^{-1}\sigma m = \sigma_0 \in T(\kappa_v). \]
There are $w\in  W$ and $b\in {B}_1(\kappa_v)$ such that 
\[ w^{-1}b^{-1}\dot{k}^{-1} m^{-1} \sigma_0 m \dot{k}bw=\sigma_0. \]
We deduce that 
\[m\dot{k}bw\in G_{\sigma_0}(\kappa_v).\]
We can assume that $wB_{\sigma_0}w^{-1}\subseteq B$ (because $w^{-1}Bw\cap G_{\sigma_0}$ is a semistandard Borel subgroup of $G_{\sigma_0}$).  
Using Bruhat decomposition for $G_{\sigma_0}(\kappa_v)$, we obtain that $\dot{k}\in m^{-1}B_{\sigma_0}(\kappa_v)wB(\kappa_v)$. 
Since $m^{-1}B_{\sigma_0}\subseteq P_{1,\sigma}$, we are done. 
\end{proof}
 Suppose $w\in W$ be an Weyl element such that $\dot{k}\subseteq P_{1,\sigma}(\kappa_v)wB(\kappa_v)$.  
Note that \[w_s P_{1,\sigma}(F_v) \subseteq w_s R(F_v)\subseteq Q(F_v)w_s,\]
therefore by Bruhat decomposition, we have as desired:
\[W^{Q}s_{w_s\eta x, v}= W^{Q}ss_{\eta x, v}.  \]

 For $H\in \ago_{M_1}$, we denote $[H]_P$ for its projection in $\ago_P$. 
 Let \[\eta x\in R(F)mN_B(\AAA)G(\mathcal{O}),\] for some $m\in M_R(\AAA)$.
  Then \[  H_P(\eta x)=[H_{M_1}(m)]_P,\]
 and for every $v\in S$, we have  \[ [s^{-1}s_{w_s\eta x, v}\xi_v]_P=[s_{\eta x, v}\xi_v]_P.  \]
They imply that 
\[ \htau_P(H_P(\eta x)+ \sum_{v\in S}s^{-1}s_{w_s\eta x,v}\xi_v)= \htau_P(H_{M_1}(m)+ \sum_{v\in S}s_{\eta x, v}\xi_v). \]
Thus, we can apply Arthur's combinatorial lemma \cite[Lemma 5.2]{Ageom}, which says that the equation \eqref{sigmaP} vanishes if 
\begin{equation}\label{HRM}
H_R(m)+ \sum_{v\in S}s_{\eta x, v}\xi_v\notin \ago_{G}+\ago_{R}^{G_\sigma}. \end{equation}

If $S\not= \emptyset$, by our definition of being generic, \eqref{HRM} always holds except when $A_{Z_{G_\sigma}}={Z_G}$, i.e. the maximal split subtorus of $Z_{G_\sigma}$ equals $Z_G$. 
However since $S\subseteq X(\mathbb{F}_q)$, the condition that $\sigma$ is conjugate to an element in $\mathcal{I}_v$ implies that  $\sigma\in T(\mathbb{F}_q)$, we deduce that $\sigma\in Z_{G}(\mathbb{F}_q)$. 

 If $S=\emptyset$, then $(e,n)=1$ and Chaudouard (\cite[Lemme 6.2.6]{Chau})  implies the same result. %These arguments establish the desired vanishing results.  
Note that there is a small typo towards the end of the proof of \cite[Lemme 6.2.6]{Chau} where the equation ``\(\sum_{j=1}^{d_i}x_{j,i}-\frac{n_i e}{n}=0\)" should read as 
``\(\sum_{j=1}^{d_i}x_{j,i}-\frac{d_in_i e}{n}=0.\)" The other arguments that follow are unchanged. 
 \end{proof}

\section{Hitchin moduli spaces}\label{B}
We introduce the Hitchin moduli spaces and moduli of Higgs bundles with parabolic structures. These spaces have been constructed respectively by N. Nitsure \cite{Nitsure} and K. Yokogawa \cite{Yokogawa}. To obtain a variety, a stability condition must be imposed. With generic parabolic weights, stability is equivalent to semistability.

\subsection{Moduli of Hitchin bundles}
In this section, we recall the notion of semistable Hitchin bundles and their coarse moduli space constructed by Nitsure \cite{Nitsure} using GIT theory. It is a variety defined over $\mathbb{F}_q$ by Nitsure's construction.  

%We only need the coprime case.  

Recall that we have chosen a canonical divisor $K_X$ over $X$, i.e., a divisor associated to the canonical line bundle or the invertible sheaf of algebraic $1$-forms. 
Let $S$ be a set of closed points of $X$. Let $D$ be the divisor \[D=K_X+\sum_{v\in S}v,\]
over $X$. A $D$-Hitchin bundle $(\mathcal{E}, \varphi)$ over $X$ consists of a vector bundle $\mathcal{E}$ together with a bundle morphism \[\varphi: \mathcal{E}\longrightarrow \mathcal{E}(D):=\mathcal{E}\otimes\mathcal{O}_X(D).  \]
We will simply call it a Hitchin pair or a Hitchin bundle if the $D$ divisor is clear in context. 
The degree and rank of $(\mathcal{E}, \varphi)$ are defined to be that of $\mathcal{E}$. 
It is called semistable if for any Hitchin subbundle $(\mathcal{F}, \varphi')$, i.e. a Hitchin bundle such that $\mathcal{F}$ is a subbundle of $\mathcal{E}$ and $\varphi'=\varphi|_{\mathcal{F}}$, we have \[\frac{\deg(\mathcal{F})}{\mathrm{rank}(\mathcal{F})}\leq \frac{\deg(\mathcal{E})}{\mathrm{rank}(\mathcal{E})}.\]
It is called stable if equality is never achieved for any non-zero proper Hitchin subbundle. It is called geometrically stable if its restriction to $\closure{X}$ is stable. 

When $(n,e)=1$, a Hitchin bundle of rank $n$ and degree $e$ is semi-stable if and only if it is stable, if and only if it is geometrically stable. In this case, Nitsure proves that there is a variety over ${\mathbb{F}}_q$ whose $\mathbb{F}_{q^k}$-points is the set of isomorphism classes of semistable $D$-Hitchin bundles over $X\otimes\mathbb{F}_{q^k}$ of rank $n$ and degree $e$ by constructing the coarse moduli space of $D$-Hitchin bundles.  
We denote it by \[\mathfrak{M}_n^{e}(D). \] 
We observe that the variety  $\mathfrak{M}_n^{e}(D)$ depends only on the line bundle associated with the divisor $D$.

\subsection{Moduli space of parabolic Higgs bundles}\label{parabolic}
In this section, we present the coarse moduli space of semistable quasi-parabolic Higgs bundles constructed by Yokogawa in \cite{Yokogawa0}. In fact, we only introduce its set of $\mathbb{F}_q$-points. 
Our exposition follows the approach in \cite[Section 4]{HS} for flagged principal bundles. For a more comprehensive account of the stack version without imposing the stability condition, see \cite[Section 2]{Yun}. We define semistability using torsors and root systems. This enables us to compare with the Arthur-Selberg trace formula in the language of adèles.

Let $S\subseteq X(\mathbb{F}_q)$ be a set of $\mathbb{F}_q$-rational points.  
We identify it with a set of closed points. As before, we set $D=K_X+\sum_{x\in S}x$. 
A quasi-parabolic Higgs bundle of rank $n$ over ${X}$ with quasi-parabolic structures in $S$ is a triple \[(\mathcal{V}, \varphi, (i_x)_{x\in {S}})\] where (1) $(\mathcal{V}, \varphi)$ is a $D$-Hitchin bundle of rank $n$, that is, a vector bundle $\mathcal{V}$ with a bundle morphism \[ \varphi: \mathcal{V}\longrightarrow \mathcal{V}(D):= \mathcal{V}\otimes\mathcal{O}_{{X}}(D). \] 
(2) For each point $x\in {S}$, $i_x$ is a flag in the $\mathbb{F}_q$-vector space $\mathcal{V}_x$. Here $\mathcal{V}$ is the fiber of $\mathcal{V}$ in $x$, and a flag is a strictly increasing chain of vector spaces of length $n$:
\[i_x: 0=L_0\subsetneq L_1\subsetneq L_2\subsetneq \cdots\subsetneq L_n= \mathcal{V}_x.\]
We require that the fiber map $\varphi_x$ of $\varphi$ in $x$ be compatible with $i_x$ for every point $x\in S$, meaning that
\[\varphi_x(L_i)\subseteq L_i, \quad \forall i\geq 1.  \]
Here, the space $L_i$ determines a $\mathbb{F}_q$-linear subspace of $\mathcal{V}(D)_x$, still denoted by $L_i$. If moreover, $\varphi_{x}(L_i)\subseteq L_{i-1}$ for all $x\in S$ and all $i$, we call it a parabolic Higgs bundle with parabolic structures in $S$. 
%We will call them (quasi-)parabolic Higgs bundles if the set $S$, hence $D$, is clear from the context. 

%Suppose $V$ is a  $n$-dimensional vector space over the finite field $\mathbb{F}_q$. We can and will identify the group $G$ with the general linear group $GL(V)$ by choosing a basis. Under this identification, the Lie algebra $\ggg$ is identified with the space of endomorphisms $End(V)$,

%\marginpar{ok?}
We can rephrase the above definitions in terms of (right) $G$-torsors. 
A vector bundle  $\mathcal{V}$ over $X$ of rank $n$ corresponds to the right $G$-torsor $\mathcal{E}:=\underline{\mathrm{Isom}}(\mathcal{O}_X^n, \mathcal{V})$. 
 For a point $x\in S$, a flag $i_x$ of the vector space $\mathcal{V}_x$ corresponds to a reduction of the structure group $\mathcal{E}_{x, B}$ of $\mathcal{E}_{x}$ to the Borel subgroup $B$ of $G$. 
A section \[ \varphi \in H^{0}(X, \ad(\mathcal{E})\otimes\mathcal{O}_X(D))\] can be identified with a morphism of vector bundles $\mathcal{V}\rightarrow \mathcal{V}\otimes\mathcal{O}_X(D)$. Here \[ \Ad(\mathcal{E})=\mathcal{E}\times^{G} \ggg:=\mathcal{E}\times \ggg/G\] with $G$ acting by the adjoint action on $\ggg$. 
 If the Hitchin morphism preserves the flag, the section takes value in $\mathcal{E}_{x, B}\times^{B}\bbb$. If, in addition, the morphism shifts the flag by degree 1 (i.e., is nilpotent at this point with respect to the flag structure), then the section takes a value in $\mathcal{E}_{x, B}\times^{B}\nnn$. A filtration 
\[0=\mathcal{F}_0\subsetneq\mathcal{F}_1\subsetneq \ldots \subsetneq \mathcal{F}_r=\mathcal{V} \]
of vector sub-bundles corresponds to a reduction of the structure group $\mathcal{E}$ from $G$ to the standard parabolic subgroup $P$ corresponding to the following partition of $n$: \[(\mathrm{rank}(\mathcal{F}_1/\mathcal{F}_0), \ldots, \mathrm{rank}(\mathcal{F}_r/\mathcal{F}_{r-1})).\] 
If a global section $\varphi$ of $\ad(\mathcal{E})\otimes\mathcal{O}_X(D)$ is obtained from a global section $\varphi_P$ of $\ad(\mathcal{E}_P)\otimes\mathcal{O}_X(D)$, we say that $(\mathcal{E}_P, \varphi_P)$ is a Hitchin reduction of $(\mathcal{E}, \varphi)$ to $P$.

Now we define semistability using the language of $G$-torsors.

Let $x\in {S}$. We continue to use the above notation. 
Any $B$-equivariant trivialization $\mathcal{E}_{x, B}\xrightarrow{\sim} B$ defines an isomorphism of $\Aut_B(\mathcal{E}_{x, B})$ with $B$. The isomorphism depends on the $B$-equivariant trivialization only by an inner automorphism of $B$. So we deduce a canonical isomorphism \[X^{*}(B)\cong X^{*}(\Aut_B(\mathcal{E}_{x, B})).\] 
Let $(\mathcal{E}_P, \varphi_P)$ be a Hitchin reduction of $(\mathcal{E},\varphi)$ to a standard parabolic subgroup $P$ of $G$. Let $\mathcal{E}_{P, x}$ be the fiber of $\mathcal{E}_P$ in $x$, 
we have a canonical isomorphism \[X^{\ast}(\Aut_P(\mathcal{E}_{P, x}))\cong X^{*}(P).\]  Since $\Aut_P(\mathcal{E}_{P,x})$ is a parabolic subgroup of $\Aut_G(\mathcal{E}_{x})$, and $\Aut_B(\mathcal{E}_{x, B})$ is a Borel subgroup of $\Aut_G(\mathcal{E}_{x}),$ there is an $a\in \Aut_G(\mathcal{E}_{x})$ such that \[a^{-1}\Aut_P(\mathcal{E}_{P,x})a \supseteq  \Aut_B(\mathcal{E}_{x, B}).\] Using these isomorphisms and the element $a$, we obtain a map $i_{\mathcal{E}}: X^{*}(P)\rightarrow X^{*}(B)$ and hence a linear map: $$s_{\mathcal{E}_{P, x}}: \ago_B \longrightarrow \ago_P. $$

Let $H_P(\mathcal{E}_P)$ be the degree map defined to be the element in $\ago_P$ such that 
for any $\mu\in X^{*}(P)$, we have \[\langle \mu, H_P(\mathcal{\mathcal{E}}_P ) \rangle =  \deg \mu(\mathcal{E}_P).  \]

%Let $ {\xi}=(\xi_x)_{x\in S}$ be a family of vectors in $\ago_B$, we define \begin{equation}\label{degree}\deg_{ {\xi}}({\mathcal{E}}_P) =  \sum_{\alpha\in \Phi(N_{P},T)}\langle \alpha, H_P(\mathcal{E}_P)+ \sum_{x\in S} s_{\mathcal{E}_P, x}\xi_x \rangle. \end{equation}

\begin{definition}\label{421}
Let $\xi=(\xi_x)_{x\in S}\in (\ago_B)^S\cong (\mathbb{Q}^n)^S$. 
A quasi-parabolic Higgs bundle $(\mathcal{E}, \varphi, (i_x))$ is called $ {\xi}$-semistable if for any standard parabolic subgroup $P$ of $G$, any Hitchin's parabolic reduction $(\mathcal{E}_P, \varphi_P)$ of $(\mathcal{E}, \varphi)$, and any $\varpi\in \hat{\Delta}_P$, we have 
\begin{equation}\label{semistabilitye}
\langle \varpi, H_P(\mathcal{E}_P)+ \sum_{x\in S} s_{\mathcal{E}_P, x}\xi_x \rangle \leq 0 .\end{equation}
It is  $ {\xi}$-stable if the above inequality is strict whenever $P\neq G$. 
\end{definition}
\begin{remark}\normalfont
In fact, our definition coincides with the one given by Simpson in \cite{Simpson} and Yokogawa in \cite{Yokogawa0}, although this may not be immediately evident.
This can be seen by observing that, in our definition, it suffices to choose $P$ as the maximal standard parabolic subgroup. This is because for each $\varpi\in \hat{\Delta}_P$, there exists a parabolic subgroup $Q$ such that $\hat{\Delta}_Q=\{\varpi\}$. For a maximal parabolic subgroup $Q$, a Hitchin's reduction $(\mathcal{E}_Q, \varphi_Q)$ of $(\mathcal{E}, \varphi)$ is just a sub-Hitchin's bundle $(\mathcal{F}, \varphi')$. 
Using the explicit expression of the fundamental weights, \eqref{semistabilitye} is equivalent to 
\[  \frac{ \deg(\mathcal{F}) + \sum_{x\in S}\sum_{i\in I_x}\xi_{x, i} }{\mathrm{rank}(\mathcal{F})}\leq  \frac{ \deg(\mathcal{E}) + \sum_{x\in S}\sum_{ i=1}^n \xi_{x,i} }{n},\]
where $I_x\subseteq \{1, 2, \ldots, n\}$ is the subset of cardinalilty $\mathrm{rank}(\mathcal{F})$, consisting of all $1\leq i\leq n$ such that \[ L_i\cap \mathcal{F}_x \neq L_{i-1}\cap \mathcal{F}_x. \]  
\end{remark}%\marginpar{??need a proof? }

%\begin{definition} A quasi-parabolic Hitchin bundle $(\mathcal{E}, \varphi, (i_x))$ is called $ {\xi}$-semistable if for any standard parabolic subgroup $P$ of $G$ and any Hitchin's parabolic reduction $\mathcal{E}_P$. We have \begin{equation}\label{semistabilitye} \deg_{ {\xi}}({\mathcal{E}}_P)\leq 0.\end{equation} A parabolic Hitchin bundle is semistable if it is semistable viewed as a quasi-parabolic Hitchin bundle. \end{definition}

\begin{prop}
If $(e,\xi)$ is generic as defined in Definition \ref{adgen}, then any $\xi$-semistable quasi-parabolic Higgs bundle is automatically $\xi$-stable and geometrically $\xi$-stable. 
\end{prop}
\begin{proof}
First, we show that $\xi$-stability coincides with $\xi$-semistability. 
It suffices to show that for any $G$-torsor $\mathcal{E}$ whose associated vector bundle has degree $e$ and any reduction of the structure group $\mathcal{E}_P$ of $\mathcal{E}$ to a standard parabolic subgroup $P$ different from $G$, we have
\begin{equation}
\langle \varpi, H_P(\mathcal{E}_P)+ \sum_{x\in S} s_{\mathcal{E}_P, x}\xi_x \rangle \neq 0 .\end{equation}
If \[\langle \varpi_\alpha, H_P(\mathcal{E}_P)+ \sum_{x\in S} s_{\mathcal{E}_P, x}\xi_x \rangle = 0,\]
let $Q\supseteq P$ be a parabolic subgroup such that \( \Delta_Q=\{\varpi_\alpha\},  \) 
then the projection of $\sum_{x\in S} s_{\mathcal{E}_P, x}\xi_x$ to $\ago_Q^G$ equals the projection of $-H_P(\mathcal{E}_P)$. The latter belongs to $X_\ast(Q)^{-e}$. 

The above arguments also show that geometric $\xi$-stability coincides with geometric $\xi$-semistability. 
To show that a $\xi$-stable quasi-parabolic Higgs bundle over $X$ is geometrically $\xi$-stable, it suffices to show that it is geometrically $\xi$-semistable. 
If it is not $\xi$-semistable, the Harder-Narasimhan filtration over $\closure{X}$ is non-trivial. As the Harder-Narasimhan filtration is unique, it must be fixed by the $\Gal(\closure{\mathbb{F}}_q/\mathbb{F}_q)$-action. This shows that it descends to $X$, leading to a contradiction.  
\end{proof}

Yokogawa constructed a variety over $\mathbb{F}_q$ that is a coarse moduli space classifying the isomorphism classes of geometrically ${\xi}$-stable quasi-parabolic Higgs bundles $(\mathcal{E}, \varphi, (i_x)_{x\in {S}})$ with $\mathcal{E}$ being of rank $n$ and degree $e$. We denote this variety by \[^{q}\mathcal{M}_{n, {S}}^{e,{\xi}}.\] 
We use $\mathcal{M}_{n,{S}}^{e,{\xi}}$ for the subvariety of parabolic Higgs bundles (\cite{Yokogawa} Remark 5.1, p. 22).

\begin{remark}\label{qpw}\normalfont
Yokogawa's results hold under the assumption (as we define it) that, for any positive root $\alpha$, one has
\[1>\langle \alpha, {\xi}_x \rangle>0. \]
Explicitly, suppose $\xi_x=(\xi_{x,1}, \ldots, \xi_{x,n})\in \mathbb{Q}^n$, 
this means that
${\xi}_{x, i}>{ \xi}_{x, j}>{\xi}_{x, i}-1$, for all $i>j$. In particular, it does not include the case where for some positive root $\alpha$, we have $\langle\alpha, {\xi}_{x}\rangle=0$. However, this is not an issue because $\langle\varpi, H_P(\mathcal{E}_P)+ \sum_{x\in S} s_{\mathcal{E}_P, x}\xi_x \rangle$ takes discrete values when $\mathcal{E}_P$ varies, therefore,  in the generic case, one can slightly vary the parameter $\xi$ without affecting the stability.    
\end{remark}

\subsection{$\mathbb{G}_m$-action}

The variety $\mathcal{M}_{n, S}^{e,  {\xi}}$ admits a $\mathbb{G}_m$-action obtained from its scalar action on the Higgs fields. A Higgs field is the morphism $\varphi$ of a parabolic Higgs bundle $(\mathcal{V}, \varphi, (i_x)_{x\in {S}})$. 
The following theorem is an analogue of \cite[Theorem 2.1]{GPHS} which concerns the coarse moduli space of Higgs bundles without parabolic structures. Related materials are carefully presented in \cite[Section 3]{HH} for Hitchin moduli space without parabolic structures. In our case, $\mathcal{M}_{n, S}^{e,  {\xi}}$ is not necessarily a fine moduli space, and we are over the field $\mathbb{F}_q$ which is not algebraically closed.

\begin{theorem}\label{PGm}
Suppose that $(e,{\xi})\in \mathbb{Z}\times (\ago_B)^{S}$ is admissible and generic as defined in Definition \ref{adgen} and Definition \ref{admissible}. Then we have
 \[|\mathcal{M}_{n, S}^{e,  {\xi}}(\mathbb{F}_q)|= q^{\frac{1}{2}\dim \mathcal{M}_{n,S}^{e,  {\xi}}} |(\mathcal{M}_{n,S}^{e,  {\xi}})^{\mathbb{G}_m}(\mathbb{F}_q)|. \]
\end{theorem}
\begin{proof}
It will be convenient to use the Poincaré family for the proof. Yokogawa's GIT construction implies that the universal family exists étale locally. In fact, in the generic case, there is (see \cite[Remark 5.1]{Yokogawa}) a scheme $R^s$ equipped with a $GL_N$-action factoring through a $PGL_N$-action such that \[\mathcal{M}_{n,S}^{e,  {\xi}}  \cong R^s/PGL_N,\]
where the quotient here is the geometric quotient. Let  $[R^s/GL_N]$ be the stack quotient of $R^s$ by $GL_N$. 
Then, $[R^s/GL_N]\rightarrow \mathcal{M}_{n,S}^{e,  {\xi}}$ is a $\mathbb{G}_m$-gerbe. 
We also know that the stack quotient $[R^s/GL_N]$ is the moduli stack of $\xi$-stable parabolic Higgs bundles, and the universal family exists tautologically for $[R^s/GL_N]$. Since any $\mathbb{G}_m$-gerbe is étale locally neutral, we deduce that the universal family exists étale locally by pullback the Poincaré family along sections.

Yokogawa proves that the variety $^{q}\mathcal{M}^{e, {\xi}}_{n, S}$ is smooth (\cite[Proposition 5.2]{Yokogawa}) as long as $(e, {\xi})$ is generic. We need to prove that $\mathcal{M}_{n,S}^{e,  {\xi}}$ is smooth as well. To simplify the notation, we denote $\mathcal{M}_{n, S}^{e,  {\xi}}$ by $\mathrm{P}$. 

\begin{lemm}\label{smooth}
The variety $\mathcal{M}_{n,S}^{e,  {\xi}}$ is smooth. 
\end{lemm}
\begin{proof}[Proof of the lemma]
The tangent bundle of $\mathrm{P}$ is given by $\mathrm{R}_{k[\varepsilon]/k}\mathrm{P}$, the Weil restriction of $\mathrm{P}$ where $k[\varepsilon]=k[\epsilon]/\epsilon^2$ is the ring of dual numbers. Let $\mathcal{T}_{\mathrm{P}}$ be its tangent sheaf. Then for any open subset $U$ of $\mathrm{P}$, $$\mathcal{T}_{\mathrm{P}}(U) = \ker(\mathrm{P}(U[\varepsilon])\rightarrow \mathrm{P}(U)), $$
where $\mathrm{P}(U)$ is the set pointed by the natural inclusion $U\rightarrow \mathrm{P}$. Let $\mathfrak{V}\xrightarrow{a}\mathrm{P}$ be an étale covering so that the universal family exists over $X\times \mathfrak{V}$. We have an isomorphism \[ \mathcal{T}_{\mathfrak{V}} \cong a^{\ast}\mathcal{T}_{\mathrm{P}}.  \]
Let $(\mathcal{E}, \varphi, (i_v)_{v\in S_{\mathfrak{V}}})$ be a universal family over $\mathfrak{V}\times X$, where $i_v$ is a family of parabolic structures that is defined using coherent sheaf interpretation of flag structures (for details see \cite[Section 1]{Yokogawa}). 
  Let $\pi: \mathfrak{V}\times X\rightarrow  \mathfrak{V}$ be the projection morphism. 
The deformation arguments in \cite[2.3]{infinitesimal} show that we have an isomorphism of locally free sheaves \[\mathcal{T}_{\mathfrak{V}} \longrightarrow R^1\pi_{\ast} C^{\bullet},  \]
where \[C^\bullet:  \mathcal{E}nd^{1}(\mathcal{E})   \xrightarrow{\ad(\varphi)} \mathcal{E}nd^{0}(\mathcal{E})(D),  \] 
 is the complex concentrated in degree $0$ and $1$,  $\mathcal{E}nd^{1}(\mathcal{E})$ is the sheaf of endomorphisms of $\mathcal{E}$ that preserve the quasi-parabolic structure, $\mathcal{E}nd^{0}(\mathcal{E})$ is its subsheaf that preserve the parabolic structure, and $\mathrm{ad}(\varphi)$ is defined by $\mathrm{ad}(\varphi)(\theta)=\varphi. \theta- \theta. \varphi $.

The tangent space at a geometric point $x$ of $\mathfrak{V}$ is given by the hypercohomology of the complex concentrated in degree $0$ and $1$: 
 \[C^\bullet_x:  \mathcal{E}nd^{1}(\mathcal{E}_x)   \xrightarrow{\mathrm{ad}(\varphi_x)} \mathcal{E}nd^{0}(\mathcal{E}_x)(D).  \] 
The Cartan-Eilenberg  first spectral sequence says that 
\[E_1^{p,q}:= H^q(\closure{X}, C^p_x)\Rightarrow H^{n}(\closure{X}, C_x^\bullet).\] Since $H^i(\closure{X}, C_x^p)=0$, for $i\geq 2$ or $p\neq 0, 1$, we deduce that \[\dim H^{1}(\closure{X}, C_x^\bullet) = \dim E_2^{1,0}+\dim E_2^{0,1}. \]
We have \[E_2^{1,0}\cong \mathrm{coker}\left(H^{0}(\mathcal{E}nd^{1}(\mathcal{E}_x)  )\rightarrow 
H^0( \mathcal{E}nd^{0}(\mathcal{E}_x)(D)) \right), \]
and \[E_2^{0,1}\cong \mathrm{ker}\left(H^{1}(\mathcal{E}nd^{1}(\mathcal{E}_x)  )\rightarrow 
H^1( \mathcal{E}nd^{0}(\mathcal{E}_x)(D)) \right).   \]

The trace map identifies $\mathcal{E}nd^{0}(\mathcal{E}_x)(\sum_{v\in \closure{S}}v)$ with \[\mathcal{E}nd^{1}(\mathcal{E}_x)^\vee.  \]
In fact, this can be checked locally over $X$ and at a point $v\in \closure{S}$ it reduces to fact that the orthongonal complement of the Lie algebra $\mathfrak{n}$ under the trace pairing is $\mathfrak{b}$.  
By Serre duality we have \[E_2^{0,1}\cong (E_2^{1,0})^{\vee}.\]  
Therefore \[\begin{split} &\dim H^{1}(\closure{X}, C_x^\bullet)  \\
= &2\dim E_2^{0,1}= 2\dim  \mathrm{coker}\left(H^{0}(\mathcal{E}nd^{1}(\mathcal{E}_x)  )\rightarrow 
H^0( \mathcal{E}nd^{0}(\mathcal{E}_x)(D)) \right). \end{split} \]

Note that any $\theta \in \mathrm{ker}(H^{0}(\mathcal{E}nd^{1}(\mathcal{E}_x)  )\rightarrow 
H^0(\mathcal{E}nd^{0}(\mathcal{E}_x)(D)) )$ defines an endomorphism of the parabolic Higgs bundle $(\mathcal{E}_x,\varphi_x, (i_v)_{v\in \closure{S}})$. Since $(\mathcal{E}_x,\varphi_x, (i_v)_{v\in \closure{S}})$ is stable, $\theta$ can only be a scalar morphism. Therefore, the kernel has dimension $1$ and 
\[\dim E_2^{0,1}= 1-h^0(\mathcal{E}nd^{1}(\mathcal{E}_x))+h^0( \mathcal{E}nd^{0}(\mathcal{E}_x)(D))  . \]
Since $\mathcal{E}nd^{0}(\mathcal{E}_x)(D)\cong \mathcal{E}nd^{1}(\mathcal{E}_x)^\vee(K_X)$, we deduce by Riemann-Roch theorem that 
\[\dim E_2^{0,1} = n^2( g-1 ) - \deg \mathcal{E}nd^{1}(\mathcal{E}_x)   +1  . \]
To calculate the degree of $\mathcal{E}nd^{1}(\mathcal{E}_x)$, we can use the short exact sequence:
\[0\longrightarrow  \mathcal{E}nd^{1}(\mathcal{E}_x)\longrightarrow  \mathcal{E}nd(\mathcal{E}_x)\longrightarrow  \mathcal{E}nd(\mathcal{E}_x)/\mathcal{E}nd^{1}(\mathcal{E}_x)\longrightarrow  0.    \]
Note that $\mathcal{E}nd(\mathcal{E}_x)/\mathcal{E}nd^{1}(\mathcal{E}_x)$ is a skyscraper sheaf supported in the set of points where we have imposed a parabolic structure. By definition of $\mathcal{E}nd^{1}(\mathcal{E}_x)$, we see that $\mathcal{E}nd(\mathcal{E}_x)/\mathcal{E}nd^{1}(\mathcal{E}_x)$ has degree $\frac{n^2-n}{2}|\closure{S}|$. Since $\mathcal{E}nd(\mathcal{E}_x)\cong \mathcal{E}_x\otimes\mathcal{E}_x^\vee$ has degree $0$, we obtain that 
\[\dim H^{1}(\closure{X}, C_x^\bullet)=2n^2(g-1)+2+({n^2-n})|\closure{S}|. \]
Therefore, the tangent dimension of $\mathcal{M}_{n, S}^{e,\xi}$ is constant. As $\mathcal{M}_{n, S}^{e,\xi}$ is of finite type and reduced, it is smooth. 
\end{proof}

We come back to the proof of the theorem. We follow the general strategy given in \cite[3]{HH} or \cite[2.1]{GPHS}. 
We have a parabolic Hitchin fibration:
\[^{q}\mathcal{M}_{n, S}^{e, {\xi}}\longrightarrow \mathcal{A}_D.\]
This is $\mathbb{G}_m$-equivariant and is projective if $(e, {\xi})$ is generic  (\cite[Proposition 5.12]{Yokogawa0}). It defines the Hitchin morphism $\mathrm{P}\rightarrow \mathcal{A}_D$, which is $\mathbb{G}_m$-equivariant and the weights of the $\mathbb{G}_m$-action are strictly positive over $\mathcal{A}_D$. Hence the $\mathbb{G}_m$-action contracts the points in $\mathrm{P}$ to the fiber over $0$. By Lemma \ref{smooth}, $\mathrm{P}$ is also smooth. We may apply the Bialynicki-Birula-Hesselink decomposition theorem (see \cite[Section 4]{B-B} over an algebraically closed field and \cite[2.10]{HH} in general). 
We obtain a disjoint union of smooth connected varieties \begin{equation}\mathrm{P}^{\mathbb{G}_m}=  \bigcup F_i .  \end{equation}
This induces a disjoint union: 
\begin{equation}\label{Bunion}\mathrm{P}= \bigcup F_i^{+},  \end{equation}
where $F_i^{+}$ consists of those points such that  $\lim_{t\rightarrow 0}t.x\in F_i$ which is a locally closed subscheme of $\mathrm{P}$. 
Bialynicki-Birula's theorem further says that the morphism 
\[F_i^{+}\rightarrow F_i,  \]
is a Zariski local fibration in affine spaces.  
It now suffices to show that the fiber of the morphism has dimension $\frac{1}{2}\dim \mathrm{P}$. For this, it suffices to show that the subspace $T_x^{+}$ of the positive weights of the tangent space of $T_x$ at a point $x$ in $\mathrm{P}^{\mathbb{G}_m}$ has dimension $\frac{1}{2}\dim \mathrm{P}$.

For the statement concerning the dimension, we can pass to $\closure{\mathbb{F}}_q.$  
Then we may use \cite[Theorem, p.492]{B-B}. In fact, it suffices to show that $\mathrm{P}$ admits a quasi-symplectic form $\omega\in \Omega^2(\mathrm{P})$ which is $\mathbb{G}_m$-equivariant of weight $1$ in the sense that 
\begin{equation}\label{B4}\lambda^{*}\omega = t \otimes \omega, \end{equation}
where $\lambda: \mathbb{G}_m\times  \mathrm{P}\rightarrow  \mathrm{P}$ is the $\mathbb{G}_m$-action, and $t=Z\otimes 1$ is the regular function over $\mathbb{G}_m\times P=\Spec(k[Z, Z^{-1}])\times P$.
In fact, admitting the existence of such a quasi-symplectic form, let $x\in (\mathrm{P})^{\mathbb{G}_m}$, then $\mathbb{G}_m$ acts on the tangent space $T_{x}\mathrm{P}$. Let $X,Y\in T_{x}\mathrm{P}$ in the $\mathbb{G}_m$-eigenspace of weights $v_1, v_2$, then
\[\lambda^{*}\omega(X,Y)=\omega(\lambda. X, \lambda. Y)= \lambda^{v_1+v_2}\omega(X,Y). \]
This shows that $\omega(X,Y)= 0$ unless $v_1+v_2=1$. In particular, $T^+_x$ is Lagrangian. Hence its dimension is half that of $\mathrm{P}$. 

The required quasi-symplectic form on $\mathrm{P}$ has been constructed in \cite[6.1, 6.2, 4.5, 4.6]{infinitesimal}. However, the article \cite{infinitesimal} is written over $\mathbb{C}$, and we do not have a Poincaré family. For the reader's convenience, we outline their constructions and, more importantly, verify the condition \eqref{B4}.  

We use the notation of proof of Lemma \ref{smooth}. 
Observe that the trace map $\theta\mapsto \Tr(\theta. \cdot )$ defines an isomorphism \[\mathcal{E}nd^{0}(\mathcal{E})(\sum_{v\in {S}}v)\cong \mathcal{E}nd^{1}(\mathcal{E})^{\vee},\] 
 we then have an isomorphism of the complex $(C^\bullet)^\vee\otimes \pi^{\ast}\omega_X$ to $C^\ast[-1]$. Therefore we can apply the Grothendieck-Serre duality to deduce 
\[R^1\pi_\ast C^\bullet \times  R^{1}\pi_\ast(C^\bullet) \xrightarrow{\sim}   R^1\pi_\ast C^\bullet \times R^0\pi_\ast (C^\bullet \otimes \pi^{\ast}\omega_{X})\longrightarrow R^1\pi_{\ast}(\pi^{\ast}\omega_{X}) \xrightarrow{\sim}\mathcal{O}_{\mathfrak{V}}.  \]
This is a non-degenerate bilinear form on $R^{1}\pi_\ast(C^\bullet)$. Note that this defines a 
non-degenerate bilinear form on $\mathcal{T}_{\mathrm{P}}$, because $C^\bullet$ descends canonically to a complex on $\mathrm{P}\times X$. In fact, the universal family itself does not necessarily descend to $\mathrm{P}\times X$ as it is not unique up to a unique isomorphism. Two isomorphisms may differ by a scalar. However, such scalar isomorphisms become identity on $C^\bullet$.

Let us verify that the pairing is quasi-symplectic (even in characteristic $2$).
We can do it point-wise on $\mathfrak{V}$.  Let $x\in \mathfrak{V}(\closure{\mathbb{F}}_q)$ and consider 
 \begin{equation}\label{B5}\begin{split}
 \wedge: H^{1}(\closure{X}, C_x^\bullet)\otimes H^{1}(\closure{X}, C_x^\bullet) \xrightarrow{(\Id, \Tr)}H^{1}(\closure{X}, C_x^\bullet)\otimes H^{0}(\closure{X}, (C_x^\bullet)^\vee\otimes \omega_X) \\
 \overset{\cup}{\longrightarrow}H^{1}(\closure{X}, \omega_X) . \end{split}
   \end{equation}
By choosing a finite affine covering $\mathfrak{U}=(U_i)_i$ of $\closure{X}$,  we can identify the cohomology with \v{C}ech cohomology. Now we can do an explicit cocycle calculation. Let 
\begin{equation}\label{cechc}
(t_i, s_{ij})\in \check{C}^{1}(\mathfrak{U}, C_x^\bullet) =\check{C}^{0}(\mathfrak{U}, \mathcal{E}nd^{0}(\mathcal{E}_x)(D))\oplus \check{C}^{1}(\mathfrak{U}, \mathcal{E}nd^{1}(\mathcal{E}_x)) ,     \end{equation}
 represents a $1$-cohomological class, which implies in particular that for any indices $i,j$, we have \begin{equation}\label{B6}\varphi s_{ij}- s_{ij}\varphi= (t_i- t_j) |_{U_{ij}}.\end{equation} Then we can calculate the product in \eqref{B5}, which gives 
 \begin{equation}\label{expl}[(t_i, s_{ij})] \wedge [(t'_i, s'_{ij})] = [\Tr(t'_is_{ij})-\Tr(s'_{ij}t_j)] .  \end{equation} Therefore, the condition \eqref{B6} implies that \[[(t_i, s_{ij})] \wedge [(t_i, s_{ij})]= [\Tr((t_i-t_j)s_{ij})]=0. \]

Finally, let us verify that the quasi-symplectic form constructed above is $\mathbb{G}_m$-equivariant of weight $1$. We can still use the above \v{C}ech cocycle calculations. Note that $\mathbb{G}_m$ acts only on the Higgs fields. From the explicit presentation of  \cite[2.3]{infinitesimal}, push forward by an element $t\in \mathbb{G}_m(k)$ sends a tangent vector corresponding to the cohomological class represented by $(t_i, s_{ij})$ in \eqref{cechc} to $(tt_i, s_{ij})$. Then we read from the explicit formula \eqref{expl} that the quasi-symplectic form is $\mathbb{G}_m$-equivariant of weight $1$.
\end{proof}

\section{Finalizing the proof of the main results}\label{F5}

Recall that we have linked the number $|E_n^{irr}(\mathfrak{R})^{\Frob^{*}}|$ to a trace formula in Theorem \ref{Einfinitesimal}. To finish the proof of the main theorems, we need to link this trace formula to the number of $\mathbb{F}_q$-points of the moduli space of Hitchin bundles. 

%\subsection{$\xi$-semi-stability}

\subsection{Geometric interpretation of an infinitesimal trace formula}
The main goal of this section is to link the number of certain Hitchin pairs or parabolic Higgs bundles to an infinitesimal trace formula. 

We must first introduce $\xi$-semi-stability for elements in $G(\AAA)\times \ggg(F)$. It is an adèlic version that generalizes $\xi$-semi-stability for parabolic Higgs bundles. Without $\ggg(F)$-part, it has been studied in \cite[Section 6.1]{Yu2}. 

For any parabolic subgroup $R$ of G defined over $F$, there is a unique pair $(P, \delta)$ with $P$ a standard parabolic subgroup, and $\delta\in P(F)\backslash G(F)$ such that $R=\delta^{-1}P\delta$. We will
take advantage of this to denote a parabolic subgroup defined over $F$ by a pair $(P,  \delta)$
if there is no confusion.
\begin{definition}
Let $(x, \gamma)\in G(\AAA)\times \ggg(F)$. We say that a parabolic subgroup $(P, \delta)$ defined over $F$ is a Hitchin subgroup with respect to $(x, \gamma)$ if \[\Ad(\delta)(\gamma) \in \ppp(F). \]
We say that a Hitchin subgroup $(Q, \eta)$ is $\xi$-semistable for $(x, \gamma)$ if for any Hitchin subgroup $(P, \delta)$ contained in $(Q, \eta)$ we have
 \[ \langle \varpi, H_B(\delta x)+\sum_{v\in S} s_{\delta x, v} \xi_v \rangle\leq 0,\quad \forall \varpi\in \hat{\Delta}_P^Q.  \]
 
 We that $(x, \gamma)$ is $\xi$-semistable if $G$ is $\xi$-semistable for $(x, \gamma)$.
\end{definition}

\begin{definition}\label{admissible}
We say that $\xi=(\xi_v)_{v\in S}\in (\ago_{B})^{S}$ is admissible if
\[ 0\leq \langle\alpha, \xi_v\rangle \leq [\kappa_v: \mathbb{F}_q],\]
for any positive root $\alpha$. 
\end{definition}

\begin{theorem}[Harder-Narasimhan filtration]\label{Harder-Nara}
Suppose that $\xi$ is admissible. Let $(x, \gamma)\in G(\AAA)\times \ggg(F)$.  
Let $(Q,\eta)$ be a Hitchin subgroup with respect to $(x, \gamma)$. 
There is a unique Hitchin subgroup $(P,\delta)$ contained in $(Q,\eta)$ such that 
$(P,\delta)$ is semistable for $(x, \gamma)$ and \[\tau_P^Q(H_B(\delta x)+\sum_{v\in S}s_{\delta x, v}\xi_v)=1.\]  
Such a pair is called the canonical refinement of $(Q,\eta)$ for $(x, \gamma)$. 
\end{theorem}
\begin{proof}
The proof is the same as \cite[Theorem 6.2]{Yu2}. It is sufficient to replace parabolic subgroups by Hitchin subgroups everywhere in the proof.   
\end{proof}

We have fixed a divisor $D=K_X+\sum_{v\in S}v$. 
Let \[ \mathfrak{c}:=\ttt/W\cong \ggg/G\cong \Spec(\mathbb{F}_q[\ggg]^G).\] 
 We define $\mathfrak{c}_D$ to be the affine bundle 
 \[ \mathfrak{c}_D=\mathfrak{c}\times^{\mathbb{G}_m}\mathcal{O}_X(D).\] 
Since the weights of $\mathbb{G}_m$ acting on $\mathfrak{c}$ are $1, 2, \ldots, n$, we have
\begin{equation}\label{decom}\mathfrak{c}_D\cong \mathcal{O}_X(D)\oplus \mathcal{O}_X(2D) \oplus\cdots\oplus\mathcal{O}_X(nD).   \end{equation}
This makes $\mathfrak{c}_D$ a vector bundle. 
Let $\mathfrak{c}_{D,v}$ be the fiber of the vector bundle $\mathfrak{c}_D$ at the closed point $v$ of $X$. It is a $\kappa_v$-vector space. We define: 
\[ \mathcal{R}_D:=\prod_{v\in S} R_{\kappa_v/\mathbb{F}_q}\mathfrak{c}_{D,v},\]
where $R_{\kappa_v/\mathbb{F}_q}$ is the functor of restriction of scalars from $\kappa_v$ to $\mathbb{F}_q$. Let \[ \mathcal{A}_D:=H^0(X, \mathfrak{c}_D), \]
be the Hitchin base. It is an affine space over $\mathbb{F}_q$. We have an evaluation morphism:
\[\mathrm{ev}:  \mathcal{A}_D\longrightarrow \mathcal{R}_D,  \]
that sends a global section $s$ to $(s_v)_{v\in S}$, where $s_v$ is the image of $s$ in $\car_{D,v}$.  
By the residue theorem (\cite[Corollary p.155,  Theorem 4 p.157]{Tate}), the image of the map is contained in the linear subspace $\mathcal{R}^1_D$ of $\mathcal{R}_D$ consisting of elements $(s_v)_{v\in S}\in \prod_{v\in S} \mathfrak{c}_{D, v}(\mathbb{F}_q)$ such that 
\[\sum_{v\in S}\Tr_{\kappa_v/\mathbb{F}_q}(\Tr(s_v))=0. \]
Here the inner trace is the map defined by the projection $\mathfrak{c}_D\rightarrow \mathcal{O}_X(D)$. 
Composing with the Hitchin fibration $\mathfrak{M}^{e}_{n}(D)\rightarrow \mathcal{A}_D$, we obtain a morphism that we call the residue morphism:
\[\res:  \mathfrak{M}^{e}_{n}(D)\longrightarrow \mathcal{R}_D^1.\]

\begin{remark}\label{RMKS3}\normalfont
By the Riemann-Roch theorem, if $\deg S\geq 3-2g$, we know that the linear map \[ H^{0}(X, \mathcal{O}_X(iD))\longrightarrow \prod_{v\in S} R_{\kappa_v/\mathbb{F}_q}\mathcal{O}_X(iD)_{v}, \]
is surjective if $i>1$, and has codimension $1$ if $i=1$. Therefore the map $\mathcal{A}_D\rightarrow \mathcal{R}_D^1$ has codimension $1$ when $\deg S\geq 3-2g$. 
\end{remark}

For each $v\in S$, fix \[ o_v\in R_{\kappa_v/\mathbb{F}_q}\car_{D,v}(\mathbb{F}_q)\cong \car(\kappa_v).\] Let $\closure{\Omega}(o_v)\subseteq \ggg(\kappa_v)$ be the set of elements whose image in $\car(\kappa_v)$ equals $o_v$. Let \[\Omega(o_v)\subset\ggg(\mathcal{O}_v),\] be the set of elements in $\ggg(\mathcal{O}_v)$ whose image in $\ggg(\kappa_v)$ belongs to $\closure{\Omega}(o_v)$.

\begin{theorem}\label{B3}
Suppose that \[ D=\sum_{v\in |X|}n_v v.\]  (1) Let $o=(o_v)_{v\in S}\in \mathcal{R}^1_D(\mathbb{F}_q)$. For each $v$, let $h_v$ be the characterisitic function of the set $\wp_v^{-n_v}\Omega(o_v)$: 
\[ h_v=\mathbbm{1}_{\wp_v^{-n_v}\Omega(o_v)}.\] 
Let $h\in {C}_c^\infty(\ggg(\AAA))$ be the function defined by $$ h=\otimes_v h_v. $$
For any $e$ that is coprime to $n$, we have
\begin{equation}\label{Jgresidue}
J^{\ggg, e}(h) =   \frac{1}{q-1}  |\mathfrak{M}_n^{e}(o)(\mathbb{F}_q)| \end{equation}

(2) For each closed point $v$ outside $S$, let \[f_v  = \mathbbm{1}_{\wp^{-n_v}_v\ggg(\ooo_v)},   \]  and for each point $v\in S$, let
  \[f_v = \vol(\mathcal{I}_v)^{-1}\mathbbm{1}_{\wp^{-n_v}_v\mathfrak{I}_{v+}},  \]
  where  $\mathfrak{I}_{v+}$ is the pro-nilpotent radical of the standard Iwahori subalgebra $\mathfrak{I}_v$ of $\ggg(\mathcal{O}_v)$.  
 Let $f\in {C}_c^\infty(\ggg(\AAA))$  be the function defined by
 \[ f=\otimes f_v . \] 
Let $(e,\xi)\in \mathbb{Z}\times(\ago_B)^S$. 
Suppose that ${\xi}$ is admissible and $(e,\xi)$ is generic. Then we have
\begin{equation}\label{JgMnS}
J^{\ggg, e, {\xi}}( f) = \frac{1}{q-1}   |{\mathcal{M}}_{n, S}^{ e, {\xi}}(\mathbb{F}_q)|. \end{equation}
\end{theorem}
\begin{proof}
Let us prove \eqref{JgMnS}. The proof of \eqref{Jgresidue} is simpler but uses similar arguments. 
We will provide indications of the modifications required at the end of the proof. The proof is based on the inspiration from \cite[(3.10.2), 3.11]{Chau}, which does not consider any parabolic structure or residue.

We must go into details of Weil's dictionary between adèles and vector bundles, or $G$-torsors over a curve. We follow \cite[Section 1]{NgoH}. 
Suppose that $\mathcal{E}$ is a right $G$-torsor over $X$. Choose a $G$-equivariant trivialization $\iota: \mathcal{E}_{\eta}\rightarrow G_{\eta}$  over the generic point $\eta=\Spec(F)$ of $X$, and $\iota_v: \mathcal{E}_{\ooo_v}\rightarrow G_{\ooo_v}$, where $\mathcal{E}_{\ooo_v}$ is the pullback of $\mathcal{E}$ to $\Spec({\ooo_v})$. 
%Note that these trivializations exist since $\ker^{1}(F, G)=\{1\}$ by Proposition \ref{Hasse} and any $G$-torsor over $\Spec(\ooo_v)$ is trivial by the Lang theorem and the Hensel lemma. 
Using the morphism $\Spec(F_v)\rightarrow \Spec(F)$, we define a right $G$-equivariant isomorphism:  \[G_{F_v}\xrightarrow{  \iota_v^{-1}|_{\Spec(F_v)}  }\mathcal{E}_{F_v}\xrightarrow{\iota|_{\Spec(F_v)}}G_{F_v}.  \]
Such a right $G$-equivariant isomorphism $G_{F_v}\rightarrow G_{F_v}$ is given by left multiplication by an element $x_v\in G(F_v)$. 
The generic trivialization can be extended to a Zariski open subset $U$ of $X$. Hence for each $v\in U$, $x_v\in G(\ooo_v)$. As  $X-U$ is finite, $x=(x_v)_{v\in |X|}$ lies in $G(\AAA)$.

Changing the trivialization $\iota_v$ amounts to a right multiplication on $x_v$ by an element in $G(\ooo_v)$. 
Let $v \in S$. If a flag structure, i.e., a reduction of the structure group $\mathcal{E}_{v, B}$ of $\mathcal{E}_v$ to $B$ is given, we can choose the trivialization $\iota_v$ to be compatible with the flag structure. 
Namely, we require that there is a trivialization $b_{v}$ of $\mathcal{E}_{v, B}$ such that the following diagram is commutative: 
\[\begin{CD}
\mathcal{E}_{v,B}@>b_{v}>>B \\
@VVV @VVV \\
\mathcal{E}_{v}@>\iota_{v}|_{\Spec(\kappa_v)}>>G
\end{CD}.\]
%Note that the fiber $\mathcal{E}_v$ is also the special fiber of $\mathcal{E}_{\ooo_v}$. 
Two such trivializations of $\mathcal{E}_{\ooo_v}$ differ by an element in $\mathcal{I}_{v}$, the standard Iwahori subgroup. 
Forgetting the trivialization $\iota_v$, a $G$-torsor $\mathcal{E}$ equipped with a flag structure $i_{v}$ for each $v\in S$ and a generic trivialization $\iota$ is associated to an element in $G(\AAA)/G(\ooo^{S})\mathcal{I}_{S}$, where \[\mathcal{I_S}=\prod_{v\in S}\mathcal{I}_v\] and \[G(\ooo^{S})=\prod_{v \notin S}G(\ooo_v).\] 
Forgetting further the generic trivialization, we obtain an element in \[x\in G(F)\backslash G(\AAA)/G(\ooo^{S})\mathcal{I}_{S}. \]

An element $\gamma\in \ggg(F)$  defines a section of $\ad(\mathcal{E})_\eta$ at the generic point $\eta$. It extends to a rational section of $\ad(\mathcal{E})$ over $X$: $X\dasharrow \ad(\mathcal{E})$. The element $\gamma$ satisfies 
$$ \ad(x^{-1})(\gamma) \in \prod_{v}\wp_v^{-n_v}\ggg(\ooo_v),  $$ 
if and only if the associated rational section  extends to a genuine section \[X\longrightarrow \Ad(\mathcal{E})\otimes \mathcal{O}_X(D).\]
Moreover, the element $\gamma$ satisfies the condition 
\[\ad(x^{-1})(\gamma) \in \wp_{v}^{-n_v}\mathfrak{I}_{v+},  \] 
at the place $v\in S$, if and only if the section is compatible (nilpotent) with the flag structure. 
% and additionally  \[\ad(x^{-1})(\gamma) \in \wp_{v}^{-n_v}(t_v+\mathfrak{I}_{v+})   \]  for every $v\in S$ if and only if the section is sent to $o=(t_v)_{v\in S}$ via the residue morphism. 
Conversely, a section $\varphi$ of  $\ad(\mathcal{E})\otimes \mathcal{O}_X(D)$ defines an element $\gamma\in \ggg(F)$ if a generic trivialization of $\mathcal{E}$ has been chosen.

Suppose $(\mathcal{E}, (i_v)_{v\in S}, \iota, (\iota_v)_{v\in |X|})$ is a tuple consisting of a $G$-torsor $\mathcal{E}$, a flag structure $i_v$ for each $v\in S$, a generic trivialization $\iota$ and trivializations $\iota_v$ over $\Spec(\ooo_v)$. We suppose the flag structure $i_{v}$ and the trivialization $\iota_v$ are compatible for every $v\in S$. Let $P$ be a standard parabolic subgroup of $G$ and $\mathcal{E}_P$ a reduction of the structure group of $\mathcal{E}$ to $P$. 
Choose a generic trivialization $\iota_P$ of $\mathcal{E}_P$ and a trivialization $\iota_{P,v}$ of $\mathcal{E}_{P,v}$ for each place $v$. 
The pair $(\mathcal{E}_P, \iota_P)$ gives an element $p\in P(\AAA)/P(\ooo)$. 
 The trivialization $\iota_{P}$ defines a trivialization of $\mathcal{E}$, still denoted by $\iota_P$. 
 With the trivialization $\iota_P$, this parabolic reduction of $\mathcal{E}$ gives a Hitchin reduction if the associated element $\gamma\in \ggg(F)$ belongs to $\ppp(F)$. 
Two generic trivializations are differed by an element in $G(F)$, so there is a $\delta\in G(F)$ such that $\delta \circ \iota=\iota_P$. Similarly, for each place $v$ there is a $k_v\in G(\ooo_v)$ such that 
 $k_v \circ \iota_v=\iota_{P, v}$. We have $$ \delta x k^{-1}= p,    $$
where $k=(k_v)_{v\in |X|}\in G(\ooo)$.  
Let $s_{k,v}$ be a Weyl element such that $P(F_v)s_{k,v}{\mathcal{I}_{v}}= P(F_v) k_{v}{\mathcal{I}_{v}}$, we see that for any $\varpi\in  \hat{\Delta}_P$: 
\[\langle \varpi,  H_P(p) +\sum_{v\in S}s_{k,v}\xi_v  \rangle  =\langle\varpi, H_P(\mathcal{E}_P)+   \sum_{v\in S}s_{\mathcal{E}_P, v} \xi_v \rangle,  \] 
where the right-hand side is defined in the paragraphs above Definition \ref{421}. Therefore, $(\mathcal{E}, \varphi)$ is $\xi$-semistable if and only if the associated $(x, \gamma)$ is $\xi$-semistable.

We have set up a dictionary between ad\`eles and parabolic Higgs bundles. Let us translate the geometric side of the Lie algebra trace formula. The truncated kernel \( j^{\ggg, \xi}(x)\) is given by the expression:  
\[ \sum_{P\in \mathcal{P}_0}(-1)^{\dim \ago_P^G} \sum_{\delta\in P(F)\backslash G(F)} \htau_{P}(H_B(\delta x)  + \sum_{v\in S}s_{\delta x, v} \xi_v) \sum_{\gamma\in \ppp(F)} f(\Ad((\delta x)^{-1}) (\gamma) ). \]
Changing $\gamma$ to $\Ad(\delta^{-1})\gamma$, we have the equality: 
\[\sum_{\gamma\in \ppp(F)} f(\Ad((\delta x)^{-1}) (\gamma) ) =  \sum_{\substack{\gamma\in \ggg(F)\\  \Ad(\delta)(\gamma)\in \ppp(F)}} f(\Ad( x^{-1}) \gamma ). \]
%\( j^{\ggg, \xi}(x)\) equals 
%\[ \sum_{P\in \mathcal{P}_0}(-1)^{\dim \ago_P^G} \sum_{\delta\in P(F)\backslash G(F)} \htau_{P}(H_B(\delta x)  + \sum_{v\in S}s_{\delta x, v} \xi_v) \sum_{\substack{\gamma\in \ggg(F)\\  \Ad(\delta)(\gamma)\in \ppp(F)}} f(\Ad( x^{-1}) \gamma ) .\]
We can rearrange the sums in the expression of \( j^{\ggg, \xi}(x)\) so that it equals:  
\begin{equation}\label{second} \sum_{\substack{\gamma \in \ggg(F)\\  f(\Ad( x^{-1}) \gamma )=1 }} \sum_{\substack{ (P, \delta) \\  \Ad(\delta)(\gamma)\in \ppp(F)  }}  (-1)^{\dim\ago_P^G}\htau_{P}(H_B(\delta x)  + \sum_{v\in S}s_{\delta x, v} \xi_v).       \end{equation}
The inner sum in \eqref{second} can be further regrouped following canonical refinements whose existence and uniqueness are introduced in Theorem \ref{Harder-Nara}. The inner sum equals: 
\begin{equation}\label{525} \sum_{(Q, \eta) }  \sum_{\substack{ (P, \delta)   \\ (P,\delta)^{cr}=(Q,\eta)  }     }    (-1)^{\dim\ago_P^G}\htau_{P}(H_B(\delta x)  + \sum_{v\in S}s_{\delta x, v} \xi_v) ,     \end{equation}
where the first sum is taken over Hitchin subgroups with respect to $(x, \gamma)$ that are $\xi$-semistable, and the second sum is taken over Hitchin subgroups whose canonical refinement is $(Q, \eta)$. Note that since $(Q,\eta)$ is already $\xi$-semistable, $(P, \delta)$ has $(Q, \eta)$ as canonical refinement if and only if \[\tau_Q^P(H_B(\delta x)  + \sum_{v\in S}s_{\delta x, v} \xi_v)=1.\]
Therefore, the inner sum in \eqref{525} equals
\[  \sum_{ \{ P\mid   P\supseteq Q\}}  (-1)^{\dim \ago_P^G} \htau_P(H_B(\delta x)  + \sum_{v\in S}s_{\delta x, v} \xi_v)\tau_Q^P(H_B(\delta x)  + \sum_{v\in S}s_{\delta x, v} \xi_v).    \]
The expression is zero except if $Q=G$ following \cite[Proposition 1.7.2]{LabWal}. We conclude that \eqref{525} is $1$ if $(x, \gamma)$ is $\xi$-semistable and is $0$ otherwise. 
We deduce that 
\begin{equation}\label{jsemistable}
%\vol(\mathcal{I}_S)^{-1}  |\{   \gamma\in \ggg(F)\mid f(\Ad(x^{-1})\gamma )\neq 0 , \text{ and $(x, \gamma)$ is $\xi$-semistable}\}| . 
j^{\ggg, \xi}(x)=\sum_{\{\gamma\in \ggg(F)|\text{$(x, \gamma)$ is $\xi$-semistable}\}} f(\Ad(x^{-1})\gamma). 
 \end{equation}

\sloppy
Note that for any function $j$ on $G(F)\backslash G(\AAA)$ that has compact support and is right-invariant by $G(\ooo^{S})\mathcal{I}_{S}$, we have
\[\int_{G(F)\backslash G(\AAA)^e} j(x) \d x  = \sum_{x}\frac{\vol(  \mathcal{I}_S)}{| x (G(\ooo^{S})\mathcal{I}_{S})x^{-1}\cap G(F) |  } j(x)  ,   \]
where the sum over $x$ is taken over a set of representatives for the double quotient $G(F)\backslash G(\AAA)^e/G(\ooo^{S})\mathcal{I}_{S}$ and the coefficient in front of $j(x)$ in the formula above is the volume of $G(F)\backslash G(F)xG(\ooo^{S})\mathcal{I}_{S}$. Therefore, in terms of $G$-torsors, the above integral can be written as 
\begin{equation}\label{B2e}\int_{G(F)\backslash G(\AAA)^e} j(x) \d x  = \vol(  \mathcal{I}_S) \sum_{(\mathcal{E}, (i_v)_{v\in S})}\frac{1}{|\Aut( (\mathcal{E}, (i_v)_{v\in S})) | } j((\mathcal{E}, (i_v)_{v\in S}))  ,   \end{equation}
where the sum is taken over the set of isomorphism classes of vector bundles of rank $n$, degree $e$ with flag structures in $S$, and we view $j$ as a function over this set.

Now, we apply the arguments to the trace formula. Given any \[x\in G(F)\backslash G(\AAA)/G(\ooo^{S})\mathcal{I}_{S} \] corresponding to $(\mathcal{E}, (i_v)_{v\in S})$, 
the set of pairs $(P, \delta)$ with $\delta\in P(F)\backslash G(F)$ is in bijection with the set of of parabolic reductions $\mathcal{E}_P$ of $\mathcal{E}$. 
Let $H^{0}(\Ad(\mathcal{E})_D, S)^{ {\xi}}$ be the subset of sections $\varphi$ in $H^{0}(X, \Ad(\mathcal{E})_D)$ that are compatible with the flag structures and so that $(\mathcal{E}, \varphi,  (i_v)_{v\in S})$ is $ {\xi}$-semistable.  Here $\Ad(\mathcal{E})_D=\Ad(\mathcal{E})\otimes\mathcal{O}_X(D)$. 
 The outcome of equation \eqref{jsemistable} is interpreted according to the Weil dictionary as follows:
\[ j^{\ggg, \xi}(x)= \vol(\mathcal{I}_S)^{-1}  |H^{0}(\Ad(\mathcal{E})_D, S)^{ {\xi}}|. \]
Then, we apply \eqref{B2e},
\begin{equation}\label{B4e}J^{\ggg, e, {\xi}}(f) =  \sum_{(\mathcal{E}, (i_v)_{v\in S})}\frac{|H^{0}(\Ad(\mathcal{E})_D, S)^{ {\xi}}| }{|\Aut( (\mathcal{E}, (i_v)_{v\in S})) | }.\end{equation}
Consider the action of \(\Aut( (\mathcal{E}, (i_v)_{v\in S}) )  \) on \(H^{0}(\Ad(\mathcal{E})_D, S)^{ {\xi}}\). We have 
\[\frac{|H^{0}(\Ad(\mathcal{E})_D, S)^{ {\xi}}|}{|\Aut( (\mathcal{E}, (i_v)_{v\in S}) )|}= \sum_{\varphi}\frac{1}{|\Aut((\mathcal{E}, \varphi , (i_v)_{v\in S}))|} ,   \]
where the sum over $\varphi$ is taken over a set of representatives of orbits of this action. Equivalently, it is the sum over isomorphism classes of parabolic Higgs bundles that have $(\mathcal{E}, (i_v)_{v\in S})$ as the underlying parabolic $G$-torsor.
Since $(e,\xi)$ is generic, $\xi$-semistability coincides with geometrical $\xi$-stability,  we know that every automorphism $(\mathcal{E}, \varphi , (i_v)_{v\in S})$ is scalar multiplication. 
We have 
\[  |\Aut((\mathcal{E}, \varphi , (i_v)_{v\in S}))|=q-1.  \]
Therefore, \eqref{B4e} becomes: \[ J^{\ggg, e, {\xi}}(f) = \frac{1}{q-1} |\mathcal{M}_{n, S}^{e, \xi} (\mathbb{F}_q)|.    \]

Now we explain the proof of  \eqref{Jgresidue}. We set $\xi=0$ and forget the parabolic structure. In the proof, instead of considering Higgs fields compatible with parabolic structures, we consider Higgs fields with the desired residue. Consider a pair $(\mathcal{E}, \iota)$ consisting of a $G$-torsor $\mathcal{E}$ and a generic trivialization $\iota$, which is associated to an element $x\in G(\AAA)/G(\mathcal{O})$.
We have a morphism:
\[H^{0}(X, \Ad(\mathcal{E})_D)\longrightarrow  \prod_{v\in S}\car_{D, v}(\kappa_v).\]
Let  \( o=(o_v)_{v\in S}\) be a point in \(\prod_{v\in S}\car_{D, v}(\kappa_v)\). A section in $ H^{0}(X, \Ad(\mathcal{E})_D)$ gives an element $\gamma \in \ggg(F)$ satisfying
$$ \ad(x^{-1})(\gamma) \in \prod_{v}\wp_v^{-n_v}\ggg(\ooo_v) .  $$ 
Its image in $\car_{D,v}(\kappa_v)\cong (\ggg\sslash G)(\kappa_v)$ is the characteristic polynomial of the reduction mod-$\wp_v^{-n_v+1}$ of the element $\ad(x^{-1}_v)(\gamma)$. 
Let $H^{0}(X, \Ad(\mathcal{E})_D)_o^{\xi}$ denote the subset of sections $\varphi$ whose image is $o$ and such that $(\mathcal{E}, \varphi)$ is $\xi$-semistable. The proof for \eqref{JgMnS} works for \eqref{Jgresidue} if we replace $H^{0}(\Ad(\mathcal{E}_P)_D, S)^{\xi}$ by $H^{0}(X, \Ad(\mathcal{E})_D)_o^{\xi}$. 
\end{proof}

\subsection{Comparison of the trace formulas}\label{Compar}
Next, we prove the Theorem \ref{THEO} and Theorem \ref{AL}. Then  Corollary \ref{1.4} follows from Theorem \ref{PGm}. 

First, we observe that it suffices to prove the result for the case $k=1$. This is because if we replace $X$ by $X\otimes\mathbb{F}_{q^k}$, the Frobenius endomorphism of $\closure{X}$ defined using $X\otimes\mathbb{F}_{q^k}$ is simply $\Frob^k$. Moreover, we have 
\[\mathcal{R}_D\otimes\closure{ \mathbb{F}}_{q}\cong \prod_{x\in \closure{S}}\car_{D,x}, \]
and a point $o\in \mathcal{R}_D^1(\mathbb{F}_q)$ defines a point $o=(o_x)_{x\in \overline{S}}$ in \[ \mathcal{R}_D(\closure{ \mathbb{F}}_{q})\cong \prod_{x\in \closure{S}}\car_{D, x}(\closure{\mathbb{F}}_q), \]
such that the sum of eigenvalues of $o_x$ over all $x\in \closure{S}$ is zero. If the point $o$ is similar to $\mathfrak{R}$ as defined in the Introduction, then the action of $\Frob^\ast$ on $\mathfrak{R}$ is equivalent to the action of the Frobenius action on the family of roots of $o$. In particular, we see that the point $o$, when viewed as a point in $\mathcal{R}_D^1(\mathbb{F}_{q^k})$, is similar to $\mathfrak{R}$ under the action of $\Frob^k$. Therefore, by applying the case $k=1$ to the data $(X\otimes\mathbb{F}_{q^k}, S\otimes\mathbb{F}_{q^k}, \mathfrak{R})$, we obtain the desired result.

For $k=1$, we need to prove that when $o\in \mathcal{R}_D^1(\mathbb{F}_q)$ is similar to $\mathfrak{R}$, the following identity for any $e$ coprime to $n$ holds:
\begin{equation}\label{Jeee}
J^{e}(e_\rho)= (q-1) q^{-\frac{1}{2}\dim \mathfrak{M}_{n}^e(o)}J^{\ggg,e}(h),
\end{equation}
where $h$ is defined in Theorem \ref{B3}. 
If moreover $S\subseteq X(\mathbb{F}_q)$, and the ramifications are split i.e. each $\rho_v$ is a principal series for $v\in S$, then we have 
\begin{equation}\label{parJee}
J^{e, \xi}(e_\rho)= (q-1) q^{-\frac{1}{2}\dim \mathcal{M}_{n, S}^{e, \xi}}J^{\ggg,e, \xi}(f). \end{equation}
where $f$ is also defined in Theorem \ref{B3}. 
In fact, by Theorem \ref{Einfinitesimal} and Theorem \ref{B3}, the equation \eqref{Jeee} can be read as
\[|E_n^{irr}(\mathfrak{R})^{\Frob^{*}}| = q^{-\frac{1}{2}\dim\mathfrak{M}_n^{e}(o) }| \mathfrak{M}_n^{e}(o)(\mathbb{F}_{q}) |.   \]
Similarly, \eqref{parJee} implies Theorem \ref{AL}.

We prove \eqref{Jeee} first. 
We decompose $J^{e}(e_\rho)$ by the coarse expansion introduced in Section \ref{CoEx}: 
\[   J^{e}(e_\rho)=\sum_{\chi} J_{\chi}^{e}(e_\rho)   .\]
Note that the support of $e_\rho$ is contained in $G(\mathcal{O})$.
Theorem \ref{vanishing} says when $(e,n)=1$, we have
\[ J^{e}_{\chi}(e_\rho)=0, \]
unless $\chi=(t-\alpha)^{n}$ for some $\alpha\in \mathbb{F}_q^\times$.
Therefore, we have: 
\[J^{e}(e_\rho)= \sum_{\alpha\in \mathbb{F}_q^\times}J_{(t-\alpha)^n}^{e}(e_\rho). \]
Since the central character of $\rho$ is trivial on $Z_G(\mathbb{F}_q)$ (Proposition \ref{centralc}),  the function $e_\rho$ satisfies that 
\[e_\rho(zx)=e_\rho(x), \]
for any $x\in G(\AAA)$ and $z\in Z_G(\mathbb{F}_q)$.
We have 
\begin{equation}
J^{e}(e_\rho)= (q-1) J^{e}_{unip}({e_\rho}).
\end{equation}
where $J^{e}_{unip}({e_\rho})=J^{e}_{(t-1)^n}({e_\rho})$. 

Next, we need to use the Lie algebra trace formula and the following Proposition concerning an explicit computation of the Fourier transform.
%( \cite[Theorem 5.7]{Yu2}), which gives us the following identity: \[   J^{\ggg, e}(h)=q^{(1-g)n^2}    J^{\ggg, e}(\hat{h}).       \]

\begin{prop}\label{ft}
Let \[K_X=\sum_{v}d_v v,\]
be the canonical divisor that we use in the choice of Fourier transform in Section \ref{Fourier}. 
Let $v$ be a place of $F$. We have the following results.
\begin{enumerate}
\item
For the characteristic function of $\ggg(\ooo_v)$, we have
\[{\mathbbm{1}}_{\ggg(\ooo_v)}  =  q^{- d_v n^2 \deg v }  \hat{  \mathbbm{1}}_{ \wp_{v}^{-d_v} \ggg(\ooo_v)}. \] 
\item
For the characteristic function of  $\mathfrak{I}_v$, we have
\[ { \mathbbm{1}}_{ \mathfrak{I}_{v}}= q^{-{d_v n^2 \deg v }   -\frac{\deg v}{2}(n^2- n)}\hat{\mathbbm{1}}_{ \wp_{v}^{-d_v-1}{\mathfrak{I}_{v+}}}. \] 
\item (Springer's hypothesis \cite{Kazhdan}\cite{KV}) 
Let $T_v$ be a maximal torus of $G$ defined over $\kappa_v$ and $\ttt_v$ its Lie algebra. Let $\theta$ be any character of $T_v(\kappa_v)$. Let $t_v$ be a regular element in $\ttt_v(\kappa_v)$ and $o_v\in \car(\kappa_v)$ be its image.
Let $\rho=\epsilon_{\kappa_v}(G)\epsilon_{\kappa_v}(T_v)R_{T_v}^G\theta$ be the Deligne-Lusztig virtual
representation of $G(\kappa_v)$ corresponding to $T_v$ and $\theta$. 
Let $$e_{\rho}:=\begin{cases}
\Tr({\rho}(\closure{x}^{-1})), \quad x\in G(\ooo_v);\\
0, \quad x\notin G(\ooo_v); 
\end{cases}$$
where $\closure{x}$ denotes the image of $x$ under the map $G(\ooo_v)\rightarrow G(\kappa_v)$. %Let $\closure{\Omega}_t\subseteq \ggg(\kappa_v)$ the $\Ad (G(\kappa_v))$-orbits of $t$ and $\Omega_t\subseteq \ggg(\ooo_v)$ be the preimage of $\closure{\Omega}_t$ of the map $\ggg(\mathcal{O}_v)\rightarrow \ggg(\kappa_v)$. 
Then for any unipotent element $u\in \mathcal{U}_G(F_v)$, we have
$$e_{\rho}(u)= q^{-{d_v n^2 \deg v } -  \frac{\deg v}{2}(n^2-n))  }\hat{\mathbbm{1}}_{\wp_{v}^{-d_v-1}\Omega(o_v) }(u-1).  $$

\end{enumerate}
\end{prop}
\begin{proof}
The dual lattice of $\ggg(\ooo_v)$ with respect to the trace pairing $(x,y)\mapsto \psi(\Tr(xy))$ is given by
$$\ggg(\ooo_v)^{\perp}:=\{ g\in \ggg(F_v)\mid \langle g, h \rangle\in \wp_v^{-d_v}\text{ for each $h\in \ggg(\ooo_v)$}  \}, $$
and it is equal to $\wp_v^{-d_v}\ggg(\ooo_v)$. 
Let $\mathcal{B}(G)_v$ be the extended Bruhat-Tits buildings of $G(F_v)$. For a point $x\in \mathcal{B}(G)_v$, let $\ggg(F_v)_{x}$ be the corresponding parahoric subalgebra of $\ggg(F_v)$ and $\ggg(F_v)_{x+}$ be pro-nilpotent radical of $\ggg(F_v)_{x}$. 
It follows (\cite[Lemma 1.8.7(a)]{KV}) that for any point $x\in \mathcal{B}(G)_v$, 
$$ \ggg(F_v)_{x}^{\perp}=\wp_{v}^{-d_v-1}\ggg(F_v)_{x+}.$$
Based on this, the first and second statements follow from a direct calculation. 

The last statement can be reduced to Springer's hypothesis, which was first proved by Kazhdan in \cite{Kazhdan} for large characteristics. A proof of Springer's hypothesis for all the characteristics (as long as a Springer isomorphism exists) can be found in \cite[Theorem A.1]{KV}. 
\end{proof}

By definition, for any $z\in \zzz_\ggg(\mathbb{F}_q)$ and $x\in \ggg(\AAA)$, we have
\[ \hat{h}(z+x)= \int_{\ggg(\AAA)} h(y)\psi(\Tr(xy))\psi(\Tr(zy))\d y.  \]
Here $\psi$ is defined in section \ref{Fourier} as the composition, \[ \psi: \mathbb{A}\longrightarrow \frac{\mathbb{A}}{F+\prod_v\wp_v^{-d_v}\mathcal{O}_v} \overset{\sim}{\longrightarrow} \mathbb{F}_q\overset{\psi_0}{\longrightarrow}\mathbb{C}^\times. \]
Note that, $h(y)\neq 0$ if and only if $y_v\in \mathbbm{1}_{\wp_v^{-d_v-1}\Omega(o_v)}$ for all $v$. Let $a_v$ be the constant term of the characteristic polynomial $o_v\in \car(\kappa_v)$. We have \[\psi(\Tr(zy))=\prod_{v} \psi_v(\Tr(zy_v))=\prod_{v}\psi_0(\Tr_{\kappa_v|\mathbb{F}_q}(za_v)). \]
Since $o\in \mathcal{R}_D^1(\mathbb{F}_q)$, we have 
\[ \psi(\Tr(zy))=1. \]
Therefore \[ \hat{h}(z+x)= \hat{h}(x).  \]
As in the group case, using the coarse expansion and applying Theorem \ref{vanishing}, we have 
\begin{equation} J^{\ggg, e}(\hat{h}) = q J^{\ggg, e}_{unip}(\hat{h}).   \end{equation}
Here $J^{\ggg, e}_{unip}(\hat{h})=J^{\ggg, e}_{t^n}(\hat{h})$. 

Part (1) and Part (3) of Proposition \ref{ft} imply that the function \[ q^{ -(2g-2)n^2   -\frac{1}{2}(n^2-n)\deg S}\hat{h}\] has support in $\ggg(\mathcal{O})$ and coincides with $e_\rho$ on unipotent elements via the map $x\mapsto x+1$.  Here we have used the fact that $o\in \mathcal{R}_D^1(\mathbb{F}_q)$ is similar to $\mathfrak{R}_D$ so that the point $o_v$ has a representative $t_v$ in the torus used in the Deligne-Lusztig representation $\rho_v$.
By the definitions of truncated traces, we deduce that \[ J_{unip}^{e}(e_\rho)=q^{ -(2g-2)n^2   -\frac{1}{2}(n^2-n)\deg S} J^{\ggg, e}_{nilp}(\hat{h}) . \]

Applying the infinitesimal trace formula \eqref{ITF}, we get the following equality:
\begin{equation} \label{525}
J^{e}(e_\rho)= (q-1) q^{(1-g)n^2-1-\frac{1}{2}(n^2-n)\deg S}J^{\ggg,e}(h). 
\end{equation}
To finish the proof, we note that, by faithful flatness of Hitchin fibration and Remark \ref{RMKS3}, 
if $\deg S\geq 3-2g$, we have
\[ \dim \mathfrak{M}_n^{e}(o)=\dim \mathfrak{M}_n^{e}(D)-\dim \mathcal{R}_D^1= 2(g-1)n^2+2+(n^2-n)\deg S . \] 
Note that $S\neq \emptyset$, so we only exclude the case that $g=0$ and $\deg S\leq 2$. In this case, $\mathfrak{M}_n^{e}(o)$ is empty. 

Finally, Let us explain the proof of \eqref{parJee}. In this case, the Deligne-Lusztig induced representation $\rho_v$ is induced from a pair $(T, \theta_v)$, where $T$ is the maximal split torus of diagonal matrices. Let $\widetilde{\rho}$ be the character of 
\(G(\mathcal{O}^S)\mathcal{I}_S,\) that is trivial on $G(\mathcal{O}^S)$ and is the character $\theta_v: \mathcal{I}_v\rightarrow T(\mathbb{F}_q)\rightarrow \mathbb{C}^\times$ for every place $v\in S$. We have \[ \rho=\mathrm{Ind}_{G(\mathcal{O}^S)\mathcal{I}_S}^{G(\mathcal{O})} \widetilde{\rho}. \]
Therefore, for any irreducible automorphic representation $\pi$ of $G(\AAA)$, we have 
\[\Hom_{G(\mathcal{O})}(\rho, \pi)\cong \Hom_{G(\mathcal{O}^S)\mathcal{I}_S}(\widetilde{\rho}, \pi).   \]
Let \[ e_{\tilde{\rho}}=\vol(\mathcal{I}_S)^{-1} \mathbbm{1}_{G(\mathcal{O}^S)\mathcal{I}_S}\widetilde{\rho}^{-1}.  \]
We obtain as in Theorem \ref{Einfinitesimal} that \[ J^{e,\xi}(e_{\rho}) =J^{e,\xi}(e_{\widetilde{\rho}}).  \]
The rest of the proof is the same as in the previous case.

\subsection{Proof of Theorem \ref{malpha}}\label{finalsection}
After our main theorem, it suffices to prove the following theorem. The divisibility by $|\Pic_X^{0}(\mathbb{F}_{q^k})|$ is essentially due to an argument of Deligne in \cite{Deligne}. 
\begin{theorem}
Suppose that $(e,n)=1$ and $o\in \mathcal{R}_D^1(\mathbb{F}_q)$. 
%The $q$-Weil numbers of $\mathfrak{M}_n^{e}(o)$ are divisible by $q^{\frac{1}{2}\dim \mathfrak{M}^{e}(o)}$ as algebraic integers. Moreover, 
There are $q$-Weil integers $\alpha$ and integers $m_\alpha$ such that 
\[  |\mathfrak{M}^{e}_n(o)(\mathbb{F}_{q^k})|=|\Pic^{0}_X(\mathbb{F}_{q^k})|q^{\frac{k}{2}\dim \mathfrak{M}^{e}_n(o)}\sum_\alpha m_\alpha \alpha^k. \] 
The $\alpha$ that appeared in the above formula are monomials in $q$-Weil numbers of the curve $X$ and roots of unity. 
\end{theorem}
\begin{proof}
By Grothendieck-Lefscehtz fixed point formula, and Deligne's theorems \cite[Théorème 1]{Weil2} and \cite[Théorème 5.2.2]{Integrality}, we know that there are $q$-Weil integers $\gamma$ and integers $n_\gamma$ such that 
\[|\mathfrak{M}^{e}_n(o)(\mathbb{F}_{q^k})|=\sum_{\gamma} n_\gamma \gamma^k, \quad \forall k\geq 1. \] 
We first prove that the $q$-Weil numbers of $\mathfrak{M}_n^{e}(o)$ (i.e. these $\gamma$) are divisible by $q^{\frac{1}{2}\dim \mathfrak{M}_n^{e}(o)}$. Note that it is sufficient to prove that the $q^m$-Weil numbers of $\mathfrak{M}_n^{e}(o)_{\mathbb{F}_{q^m}}$ are divisible by $(q^m)^{\frac{1}{2}\dim \mathfrak{M}_n^{e}(o)}$ for some $m$. We choose $m$ divisible enough so that $S\otimes \mathbb{F}_{q^m}\subseteq X(\mathbb{F}_{q^m})$ and every ramification $\mathfrak{R}_x$ is split for $x\in \closure{S}$. 
Using Theorem \ref{THEO} and Theorem \ref{AL} proved previously,  we see that for any generic $(e', \xi)$ with $\xi$ being admissible, we have
\[  |\mathfrak{M}_n^e(o)(\mathbb{F}_{q^{km}})| = |\mathcal{M}_{n,S}^{e^{\prime}, \xi}(\mathbb{F}_{q^{km}})|, \quad \forall k\geq 1.  \]
This implies that the $q^m$-Weil numbers of $\mathfrak{M}_n^e(o)_{\mathbb{F}_{q^m}}$ coincide with those of $(\mathfrak{M}_{n,S}^{e^{\prime}, \xi})_{\mathbb{F}_{q^m}}$. Therefore, the desired divisibility is a corollary of Theorem \ref{PGm}.

Deligne (\cite[6.4, 6.5]{Deligne}) has proven a result that can be applied to show that $|\mathfrak{M}_{n}^{e}(o)(\mathbb{F}_{q^k})|$ is divisible by $|\Pic_{X}^0(\mathbb{F}_{q^k})|$.  
Let $A=\Pic_{X}^0$ and $\closure{A}$ its base change to $\closure{\mathbb{F}}_q$.  
Note that $A$ acts on $\mathfrak{M}_{n}^{e}(o)$ by tensor on the underlying vector bundle of a Hitchin bundle, and we have a morphism \[ f: \mathfrak{M}_{n}^{e}(o)\longrightarrow A\] which sends a Hitchin bundle to the determinant of its underlying vector bundle. 
This morphism is clearly equivariant for the action of $A$ on $\mathfrak{M}_{n}^{e}(o)$ defined above and the action of $A$ on itself by $\mathcal{L}_0: \mathcal{L}_1 \longrightarrow  \mathcal{L}_1\otimes \mathcal{L}^{\otimes n}_0$. 
In \cite[6.4, 6.5]{Deligne}, Deligne proves that under these hypotheses, the sheaf $R^{i}f_{!}\mathbb{Q}_\ell$ is smooth and semisimple, and there is a sheaf $\mathcal{H}^i$ over $\Spec(\mathbb{F}_q)$ such that the invariant sub-sheaf of $R^{i}f_{!}\mathbb{Q}_\ell$ under the action of ${\pi_1(\closure{A})}$ satisfies: 
\[  (R^{i}\closure{f}_{!}\mathbb{Q}_\ell)^{\pi_1({A}\otimes\closure{\mathbb{F}}_q)}\cong a^\ast \mathcal{H}^{i}|_{\closure{A}},  \]
where $a: A\longrightarrow \Spec(\mathbb{F}_{q})$ is the structure morphism. The Grothendick-Lefschetz fixed points formula implies that:
\begin{equation}\label{mneof}  |\mathfrak{M}_{n}^{e}(o)(\mathbb{F}_{q^k})| =  |\Pic_{X}^0(\mathbb{F}_{q^k})|\sum_{i}(-1)^i\Tr(F_{q}^{\ast k} | \mathcal{H}^i|_{\Spec(\closure{\mathbb{F}}_q)}   )  .\end{equation}

To complete the proof, it is necessary to demonstrate that the $q$-Weil numbers of $\mathcal{H}^i$ are divisible by $q^{\frac{1}{2}\dim \mathfrak{M}_n^e(o)}$. If there is an eigenvalue $\beta$ of $F_q^\ast$ on $\mathcal{H}^i|_{\Spec(\closure{\mathbb{F}}_q)}$ that is not divisible by $q^{\frac{1}{2}\dim \mathfrak{M}_n^e(o)}$, we can assume that $\beta$ has the minimum weight among such $q$-Weil numbers. 
Note that $ |\Pic_{X}^0(\mathbb{F}_{q^k})|=\prod_{i=1}^{2g}(1-\sigma_i^k)$ where $\sigma_i$ are $q$-Weil numbers of the curve $X$ which, for any isomorphism between $\closure{\mathbb{Q}}_\ell$ and $\mathbb{C}$,  have absolute value $q^{\frac{1}{2} }$. 
Then, \eqref{mneof} leads to a contradiction because $\beta$ would also be a $q$-Weil number of $\mathfrak{M}_{n}^{e}(o)$.

Finally, Mellit \cite[Corollary 7.9]{Mellit} has shown that the number \[ q^{-\dim \mathcal{M}_{n,S}^{e^{\prime}, \xi}/2}|\mathcal{M}_{n,S}^{e^{\prime}, \xi}(\mathbb{F}_q)|\] can be obtained by evaluating a polynomial which is independent of the curve $X$ in the Frobenius eigenvalues of the $\ell$-adic cohomology the curve $X$. 
In particular, $q$-Weil numbers of $\mathcal{M}_{n,S}^{e^{\prime}, \xi}$ are monomials in those of the curve $X$ and $q$ or $q^{-1}$. To conclude, we can apply \cite[Théorème C.1]{Yu1} which implies that they are indeed monomials in $q$-Weil numbers of the curve $X$ multiplied by $q^{\dim \mathcal{M}_{n,S}^{e^{\prime}, \xi}/2}$. 
Therefore this completes the proof of the last statement. 
\end{proof}

\appendix

\section{Proof of the existence of similar elements}\label{regulartra}
In this appendix, we show that if $p\neq 2$ and every place $v\in S$ satisfies that $|\kappa_v|\geq \frac{n(n-1)}{2}$,  then for any $(\mathfrak{R}_x)_{x\in \closure{S}}$ satisfying \eqref{13} and \eqref{prode=1}, there exists an $o\in \prod_{v\in S}\mathcal{R}_v(\mathbb{F}_q)$ similar to $(\mathfrak{R}_x)_{x\in \closure{S}}$. 

It is sufficient to consider the case that $S$ is a singleton $\{v\}$. Let $q'=|\kappa_v|$. 

Let $(n_1, n_2, \ldots, n_r)$ be a partition of $n$. We need to show that there is a monic polynomial $o_v\in \kappa_v[t]$ so that the degree of irreducible factors of $o_v$ coincides with this partition. 

We first show that there are $x_i\in \mathbb{F}_{q'}$, $i=1, \ldots, r$, such that $x_i$ are pairwise distinct and \[\sum_{i=1}^{r}x_i=0.\]
Note that since $p\neq 2$, for any $i\neq j$,
the linear subspaces defined by $x_1+x_2+\cdots +x_r=0$ and $x_i=x_j$ generate the whole space. 
Therefore their intersection is a linear subspace of $(\mathbb{F}_{q'})^r$ of dimension $r-2$. 
We see that the cardinality of the set \[ \{(x_1, \ldots, x_r)\in \mathbb{F}_{q'}^{r}\mid x_i\neq x_j (i\neq j), \text{ and } \sum x_i=0  \}\]
 is strictly larger than
\[ {q'}^{r-1}-{q'}^{r-2} \frac{r(r-1)}{2}\geq 0. \]
 This proves the existence of such $x_i$. 

We then apply the following lemma by setting $c=x_i$ and $m=n_i$ ($i=1, \ldots, r$) to construct the polynomial $o_v$. This will complete the proof. 
\begin{lemm}
Suppose that $p\neq 2$. 
For any $c\in \mathbb{F}_{q}$, there exists an irreducible monic polynomial of degree $m\in \mathbb{N}^{\ast}$ in $\mathbb{F}_q[x]$ of the form
\[x^m-cx^{m-1}+\ldots\]
\end{lemm}
\begin{proof}
There is nothing to prove when $m=1$, so we suppose $m\geq 2$. 

The lemma is equivalent to the existence of \[ \alpha\in  \mathbb{F}_{q^m}\backslash\bigcup_{d\mid m, d\neq m} \mathbb{F}_{q^{d}}  \]
such that $\Tr_{\mathbb{F}_{q^m}/\mathbb{F}_q}(\alpha)=c$. Since the trace map is an $\mathbb{F}_q$-linear surjective map, there are exactly $q^{m-1}$ elements $\alpha\in\mathbb{F}_{q^m} $ such that $\Tr_{\mathbb{F}_{q^m}/\mathbb{F}_q}(\alpha)=c$. 

If $m\neq 2$, then we have the following estimate:
\[   |\bigcup_{d\mid m, d\neq m} \mathbb{F}_{q^{d}}|\leq \sum_{i=1}^{i=[m/2]} q^{i}< q^{m-1}.\]
Therefore, there is some $\alpha$ satisfying the requirement. 

If $m=2$, then we need to show that $\ker\Tr \neq \mathbb{F}_q$. This is clear if $p\neq 2$. \end{proof}

\end{document}